\documentclass[review]{elsarticle}
\usepackage{lineno,hyperref}
\modulolinenumbers[5]
\usepackage{epsfig,epsf,rotating}
\usepackage{subfigure}
\usepackage{epsf}
\usepackage{latexsym}
\usepackage{amssymb}
\usepackage{amsthm}
\usepackage{amsmath}
\usepackage{enumerate}
\usepackage{graphicx}
\usepackage{enumitem}
\usepackage{wrapfig}
\usepackage{caption}
\usepackage{soul}
\usepackage{fancyhdr}
\usepackage[normalem]{ulem}
\usepackage{hyperref}
\usepackage{xcolor}
\usepackage{algorithm}
\usepackage{algorithmic}
\usepackage[toc,page]{appendix}
\usepackage{tikz}
\usepackage{soul}
\usepackage{geometry} 
\usepackage[export]{adjustbox} 
\usepackage{hhline} 
\usepackage{multirow} 

\newcount\colveccount
\newcommand*\colvec[1]{
        \global\colveccount#1
        \begin{bmatrix}
        \colvecnext
}
\def\colvecnext#1{
        #1
        \global\advance\colveccount-1
        \ifnum\colveccount>0
                \\
                \expandafter\colvecnext
        \else
                \end{bmatrix}
        \fi
}

\makeatletter
\renewcommand\d[1]{\mspace{6mu}\mathrm{d}#1\@ifnextchar\d{\mspace{-3mu}}{}}
\makeatother

\newcommand{\Bk}{\color{black}}

\newcommand{\bfU}{{\boldsymbol U}}
\newcommand{\bfV}{{\boldsymbol V}}
\newcommand{\bfu}{{\boldsymbol u}}
\newcommand{\bff}{{\boldsymbol f}}
\newcommand{\bfx}{{\boldsymbol x}}
\newcommand{\bfv}{{\boldsymbol v}}
\newcommand{\bfg}{{\boldsymbol g}}
\newcommand{\bfn}{{\boldsymbol n}}
\newcommand{\D}{{\mathrm D}}

\newcommand{\N}{{\mathrm N}}

\newcommand{\DG}{{\mathrm D \mathrm G}}

\newtheorem{proposition}{Proposition}
\newtheorem{lemma}{Lemma}

\begin{document}
\makeatletter
\def\ps@pprintTitle{%
  \let\@oddhead\@empty
  \let\@evenhead\@empty
  \let\@oddfoot\@empty
  \let\@evenfoot\@oddfoot
}
\makeatother
\begin{frontmatter}

\title{A Sequential Discontinuous Galerkin Method for Two-Phase Flow in Deformable Porous Media\tnoteref{mytitlenote}}

\author[riceAddress]{Boqian Shen}
\author[riceAddress]{Beatrice Riviere}

%

\address[riceAddress]{Rice University, 6100 Main St, Houston, TX 77005. Funding by NSF-DMS 1913291 is acknowledged.}

\begin{abstract}
We formulate a numerical method for solving the two-phase flow poroelasticity equations. The scheme employs
the interior penalty discontinuous Galerkin method and a sequential time-stepping method.  
The unknowns are the phase pressures and the displacement.  Existence of the solution is proved.  Three-dimensional numerical results show
the accuracy and robustness of the proposed method.
\end{abstract}


\end{frontmatter}

\section{Introduction}

The field of poromechanics pertains to the study of coupled fluid flows and mechanical deformations in porous media.  Applications include the prediction of land subsidence due to extraction of water and/or hydrocarbons from  subsurface \cite{faunt2016water}.  
Mathematical models of the poroelastic two-phase flow problem can be found in \cite{lewis1998finite} and were 
derived by Biot \cite{biot1941general,biot1956theory} using a phenomenological approach.   
In the case of single phase flow, the poroelasticity equations have been extensively studied by applied mathematicians and engineers in the scientific literature \cite{muradloula,BarryMercer1998,WheelerGai,phillips2008coupling,yi2013coupling,ChaabaneRiviere2017b}. In contrast, there are very few works on the design of efficient numerical methods for multiphase flows in deformable porous media.
The main contribution of this work is the formulation of a numerical method that employs discontinuous piecewise polynomial approximations
for the wetting and non-wetting phase pressures and the displacement of the medium.
At each time step, the mass balance equations and the momentum equation are sequentially solved.
Stabilization terms are added to the discrete momentum equation, in the same spirit as what was done in \cite{ChaabaneRiviere2017} for
single phase flow in deformable porous media.

In this work, we focus on isothermal flows where inertial forces are neglected.  
The resulting coupled
partial differential equations can be solved fully implicit, iteratively or sequentially \cite{dean2006comparison}. 
Fully implicit finite element 
methods are the most stable ones but also the most computationally expensive.
In \cite{schrefler1993fully}, finite element methods
in space are combined with the theta method in time and the resulting system is solved by Newton-Raphson's method at each time step.
The method is applied to one-dimensional and two-dimensional problems. In \cite{yang2014fully}, fully implicit mixed finite element methods combined with standard finite element methods  are
applied to solve for pressure, saturation,  displacement and their gradients in two-dimensional problems. 
The iterative approach (fixed-stress split) is combined with finite volume methods in \cite{asadi2015comparison} 
for different choices of primary unknowns and for one-dimensional problems.
Our approach is novel in the sense that no iterations are needed for stability. At each time step, each equation is solved separately and
the computational cost is smaller than the one for fully implicit methods. We apply the proposed method to three-dimensional problems
and we study the impact of heterogeneities (discontinuous capillary pressure) and loading on the propagation of the fluid phases in the medium.
Finally, we point out that fully implicit finite element method
has been applied to more complex dynamic and non-isothermal flows in \cite{lizienkiewicz1990,schrefler2001fully,gawinbaggio,khoei2020thermo}.


%
%

An outline of the paper follows. Section~\ref{sec:problem} introduces the mathematical model and the assumptions on the input
data. The numerical algorithm is described and analyzed in Section~\ref{sec:scheme}.  Numerical results, including convergence rates and
validation of the method by benchmark problems, can be found in Section~\ref{sec:numer}. Conclusions follow.

\section{Model Problem}
\label{sec:problem}
\label{sec:model}

Mathematical models for compressible two-phase flow poroelasticity are described by two mass conservation equations
coupled by a momentum conservation equation \cite{lewis1998finite}.  Let $p_w, s_w$ (resp. $p_o, s_o$) denote the wetting (resp. non-wetting) phase pressure and saturation respectively and let $\bfu$ denote the displacement of the porous medium $\Omega\subset\mathbb{R}^3$. 
By definition, $s_o=1-s_w$, and we use this relation to eliminate the non-wetting phase saturation from the system of equations. 
The difference between phase pressures is the capillary pressure, $p_c$, which is a given nonlinear function of $s_w$, according to the Brooks-Corey model \cite{brookscorey}:
\begin{equation}\label{eq:pc}
p_c = p_c(s_w) = p_o - p_w, \quad s_w =  \left(\frac{p_d}{p_c}\right)^2, 
\end{equation}
where $p_d>0$ is a constant entry pressure.
We choose for primary unknowns the phase pressures and the displacement. 
The nonlinear model coupling flow and deformation can be described by the following equations: 
\begin{align}
\mathcal{C}_1(p_o,p_w) \frac{\partial p_w}{\partial t} + \mathcal{C}_2(p_o,p_w) \frac{\partial p_o}{\partial t} 
-  \nabla \cdot (\lambda_w(s_w) K \nabla p_w)
+ \alpha s_w \frac{\partial (\nabla \cdot\bfu)  }{\partial t} = f_w,\label{eq:pb1}\\
\mathcal{C}_3(p_o,p_w) \frac{\partial p_o}{\partial t} + \mathcal{C}_4(p_o,p_w) \frac{\partial p_w}{\partial t} 
- \nabla \cdot (\lambda_o(s_w) K \nabla p_o)
+ \alpha (1-s_w) \frac{\partial (\nabla \cdot\bfu)  }{\partial t} = f_o, \label{eq:pb2}\\
-\mu \Delta \bfu - (\lambda+\mu)\nabla (\nabla \cdot \bfu) + \nabla   (s_w p_w + (1-s_w) p_o)   = \bff_\bfu. \label{eq:pb3}
\end{align}
The mass balance equations for the wetting and non-wetting phase are \eqref{eq:pb1} and \eqref{eq:pb2} respectively 
whereas \eqref{eq:pb3} represents the momentum equation for quasi-static elastic deformation of the medium. The coefficients $C_i$ are nonlinear functions of the phase pressures  (see \eqref{eq:pc}):
\begin{align}
\mathcal{C}_1(p_o,p_w) =& \frac{\alpha-\phi}{K_s} s_w^2 + \frac{\phi s_w}{K_w} + \left( \frac{\alpha-\phi}{K_s} s_w p_c - \phi\right) \frac{d s_w}{d p_c},\\
\mathcal{C}_2(p_o,p_w) =& \frac{\alpha-\phi}{K_s} s_w(1-s_w) - \left( \frac{\alpha-\phi}{K_s} s_w p_c  -\phi\right) \frac{d s_w}{d p_c},\\
\mathcal{C}_3(p_o,p_w) =& \frac{\alpha-\phi}{K_s} (1-s_w)^2 + \frac{\phi (1-s_w)}{K_o} - \left( \frac{\alpha-\phi}{K_s} (1-s_w) p_c + \phi\right) \frac{d s_w}{d p_c},\\
\mathcal{C}_4(p_o,p_w) =& \frac{\alpha-\phi}{K_s} s_w(1-s_w) + \left( \frac{\alpha-\phi}{K_s} (1-s_w) p_c +\phi\right) \frac{d s_w}{d p_c}.
\end{align}
We describe briefly the different coefficients in the equations above. The absolute permeability field $K$ and the
porosity field $\phi$ are given positive scalar functions; $K$ may be discontinuous and vary in space over several orders of magnitude.
Other input data are known constants: the Biot-Willis constant $\alpha$; the bulk moduli for the solid structure
and the fluid phases, $K_s, K_w, K_o$;  the Lam\'e parameters $\lambda, \mu$; and the phase viscosities $\mu_w$ and $\mu_o$. 
The phase mobilities, $\lambda_w, \lambda_o$,  are the ratios of
the phase relative permeability $k_{ri}$ to the phase viscosity $\mu_i$ and they are given functions of the saturation:
\begin{equation}\label{eq:relperm}
\lambda_i(s_w) = \frac{k_{ri}(s_w)}{\mu_i}, \quad i=w,o, \quad  k_{rw}(s_w) = s_w^4,  \quad k_{ro}(s_w) = (1-s_w)^2(1-s_w^2).
\end{equation}
The Biot-Willis constant $\alpha$ is close to $1$. 
For realistic porous media with porosity less than  $0.5$, this implies that
the quantity $(\alpha - \phi)$ is non-negative. 
The porous medium is such that the  bulk modulus for
the solid is much larger than the capillary pressure, and thus we assume that
\[
\frac{p_c}{K_s} << 1.
\]
This implies that 
\[
\frac{\alpha-\phi}{K_s} s_w p_c - \phi \leq 0.
\]
From \eqref{eq:pc}, we see that the derivative $s_w'(p_c)$ is negative. 
Therefore, with the assumptions above, we can determine the sign of two of the scalar functions $\mathcal{C}_i(p_o,p_w)$.
\begin{equation}\label{eq:nonneg}
\mathcal{C}_1(p_o,p_w) \geq 0, \quad \mathcal{C}_3(p_o,p_w) \geq 0.
\end{equation}
This motivates the use of a sequential scheme where \eqref{eq:pb1} is solved for $p_w$ and \eqref{eq:pb2} is solved for $p_o$.
The equations \eqref{eq:pb1}-\eqref{eq:pb3} are completed by initial and boundary conditions.
\begin{eqnarray}
p_w & = & p_{w}^0,\quad \mbox{in} \quad \Omega\times \{0\},\\
p_o & = & p_{o}^0,\quad \mbox{in} \quad \Omega\times \{0\},\\
\bfu & = & \bfu^0,\quad \mbox{in} \quad \Omega\times \{0\}.
\end{eqnarray}
The boundary of the medium is decomposed into Dirichlet and Neumann parts for pressures and displacement:
\[
\partial\Omega = \Gamma_{p\D}\cup\Gamma_{p\N} = \Gamma_{\bfu\D}\cup\Gamma_{\bfu\N}.
\]
Boundary data are prescribed by the following conditions: 
\begin{align}
p_w = & p_{w\D}, \quad p_o = p_{o\D}, \quad \mbox{on}\quad \Gamma_{p\D}\times (0,T),\\
\lambda_w(s_w) K \nabla p_w\cdot\bfn  = & g_{w}, \quad \lambda_o(s_w) K \nabla p_o \cdot\bfn = g_o, \quad \mbox{on}\quad \Gamma_{p\D}\times (0,T),\\
\bfu = & \bfu_\D, \quad \mbox{on}\quad \Gamma_{\bfu\D}\times (0,T),\\
\mu \nabla \bfu \, \bfn + (\lambda+\mu) (\nabla \cdot\bfu)\bfn
 = & \bfg_\bfu, \quad \mbox{on}\quad \Gamma_{\bfu\N}\times (0,T).
\end{align}

\section{Discontinuous Galerkin Scheme}
\label{sec:scheme}

The equations are discretized by the interior penalty discontinuous Galerkin method.  Let $\mathcal{E}_h$ be a partition of the domain
made of tetrahedral elements of maximum diameter $h$.  Let $\Gamma_h$ denote the set of interior faces. For any interior face $e$, we fix
a unit normal vector $\bfn_e$ and we denote by $E_e^1$ and $E_e^2$ the two tetrahedra that share the face $e$ such that the vector
$\bfn_e$ points from $E_e^1$ into $E_e^2$. The jump and average of a function $q$ across an interior face $e$ are denoted by
$[q]$ and $\{q\}$ respectively:
\[
[q] = q|_{E_e^1}-q|_{E_e^2}, \quad \{q\}=\frac12 \left( q|_{E_e^1}+q|_{E_e^2}\right), \quad \forall e=\partial E_e^1\cap\partial E_e^2.
\]
The jump and average of $q$ on a boundary face are, by convention, equal to the trace of $q$:
\[
[q] = q|_{e}, \quad \{q\}=q|_e, \quad \forall e\subset\partial\Omega.
\]
The DG spaces, denoted by $Q_h$ and $\bfV_h$, consist of discontinuous piecewise linears:
\[
Q_h=\{ q\in L^2(\Omega): \, q|_E \in \mathbb{P}_1(E), \, \forall E\in\mathcal{E}_h\}, \quad
\bfV_h = Q_h\times Q_h\times Q_h.
\]
We denote by $\Pi$ the cut-off operator that restricts any function $q$ to the interval $[0,1]$. The parameter $\epsilon$ is chosen equal
to $10^{-8}$ in our numerical results.
\[
\Pi(q)(\bfx) = \left\{
\begin{array}{lr}
1-\epsilon  \Bk & \mbox{if } q(\bfx) > 1-\epsilon,\\
q(\bfx) & \mbox{if } 0\leq q(\bfx) \leq 1,\\
\epsilon \Bk & \mbox{if } q(\bfx) < \epsilon.
\end{array}
\right.
\]
Let $0=t_0<t_1<\dots<t_N=T$ be a partition of the time interval $(0,T)$. For reasons that will be apparent below, we choose two time step values $\tau_0$ and $\tau$ and we define
\[
t_1 = \tau_0, \quad t_n = t_1+(n-1)\tau, \quad \forall n\geq 2.
\]
Let $P_w^n, P_o^n$ and $\bfU^n$ denote the DG approximations of $p_w, p_o$ and $\bfu$ evaluated at time $t_n$. We define
\begin{equation}
S_w^n = \Pi(p_c^{-1}(P_o^n-P_w^n)), \quad \forall n\geq 1.
\label{eq:truncatedsat}
\end{equation}
The scheme consists of three sequential steps for $n\geq 1$:\\
Step 1: Given $P_w^n\in Q_h$, $P_o^n, P_o^{n-1}\in Q_h$ and $\bfU^n,\bfU^{n-1}\in \bfV_h$, find $P_w^{n+1}\in Q_h$ such that
\begin{align}
\left(\mathcal{C}_1(P_o^n,P_w^n) \frac{P_w^{n+1}-P_w^n}{\tau} + \mathcal{C}_2(P_o^n,P_w^n) \frac{P_o^n-P_o^{n-1}}{\tau}, q_h\right)_\Omega
+ a(\lambda_w^n K; P_w^{n+1},q_h) \nonumber\\
+ \alpha  b_\bfu(S_w^n;\frac{\bfU^{n}-\bfU^{n-1}}{\tau},q_h) = \ell_w(t_{n+1};q_h),\quad \forall q_h\in Q_h.\label{eq:discpb1}
\end{align}
Step 2: Given $P_o^n\in Q_h$, $P_w^n, P_w^{n+1}\in Q_h$ and $\bfU^n, \bfU^{n-1}\in \bfV_h$, find $P_o^{n+1}\in Q_h$ such that
\begin{align}
\left(\mathcal{C}_3(P_o^n,P_w^n) \frac{P_o^{n+1}-P_o^n}{\tau} + \mathcal{C}_4(P_o^n,P_w^n) \frac{P_w^{n+1}-P_w^n}{\tau},q_h \right)_\Omega
+ a(\lambda_o^n K; P_o^{n+1},q_h) \nonumber\\
+ \alpha b_\bfu(1-S_w^n;\frac{\bfU^{n}-\bfU^{n-1}}{\tau},q_h)  = \ell_o(t_{n+1};q_h), \quad \forall q_h\in Q_h.
\label{eq:discpb2}
\end{align}
Step 3: Given $P_o^{n+1}, P_w^{n+1} \in Q_h$ and $\bfU^{n}, \bfU^{n-1}\in\bfV_h$, find $\bfU^{n+1}\in\bfV_h$ such that
\begin{align}
c(\bfU^{n+1},\bfv_h) + b_p(S_w^{n+1} P_w^{n+1}+(1-S_w^{n+1}) P_o^{n+1},\bfv_h) 
+\gamma \left(\frac{\bfU^{n+1}-\bfU^n}{\tau },\bfv\right)_\Omega
\nonumber\\
-\gamma \left(\frac{\bfU^{n}-\bfU^{n-1}}{\tau },\bfv\right)_\Omega
= \ell_\bfu(t_{n+1};\bfv_h), \quad
\forall \bfv_h\in\bfV_h.  \label{eq:discpb3}
\end{align}
In \eqref{eq:discpb1}, \eqref{eq:discpb2}, the coefficients $\lambda_w^n, \lambda_o^n$ are the functions $\lambda_w$
and $\lambda_o$ evaluated at $S_w^n$. In \eqref{eq:discpb3}, the parameter $\gamma$ is a positive constant that is user-specified and
that multiplies a stabilization term involving the discrete displacements.  The numerical scheme \eqref{eq:discpb1}-\eqref{eq:discpb3} is sequential as the flow and displacement equations are solved separately. However, each equation is solved implicitely with respect to its primary unknown ($P_w^{n+1}$ for \eqref{eq:discpb1}, $P_o^{n+1}$ for \eqref{eq:discpb2} and $\mathbf{U}^{n+1}$ for \eqref{eq:discpb3}). One novel contribution of this work is the use of the stabilization term that multiplies $\gamma$; this term is required for convergence of the method. For single-phase flow in deformable porous media, stability and convergence of the scheme is obtained if $\gamma$ is sufficiently large \cite{ChaabaneRiviere2017}.  The convergence proof for the case of two-phase flow in deformable porous media remains an open question.

The $L^2$ inner-product over $\Omega$ is denoted by $(\cdot,\cdot)_\Omega$. Similary, we use the notation
$(\cdot,\cdot)_E$ and $(\cdot,\cdot)_e$ for the $L^2$ inner-product over an element $E$ and a face $e$. We now describe the forms $a(\cdot;\cdot,\cdot), b_\bfu(\cdot;\cdot,\cdot)$, 
$c(\cdot,\cdot), b_p(\cdot,\cdot)$ that correspond to the discretizations of the differential operators in the mathematical model.
For the operator of the form $ \chi \nabla \cdot \bfu$ with $\chi$ being a scalar-valued function, we propose the following discretization:
\[
b_\bfu(\chi;\bfu,q) = -\sum_{E\in\mathcal{E}_h} (\bfu,\nabla (\chi q))_E
+ \sum_{e\in\Gamma_h\cup\partial\Omega} (\{\bfu\cdot\bfn_e\}, [\chi q])_e.
\]
For the operator of the form $\nabla  q$, we apply the following discretization:
\[
b_p(q,\bfv) = \sum_{E\in\mathcal{E}_h} (\nabla q, \bfv)_E
- \sum_{e\in\Gamma_h} ([q],\{\bfv\cdot\bfn_e\})_e.
\]
For the operator of the form $-\nabla\cdot(\chi\nabla p)$ with $\chi$ being a scalar-valued function, we utilize the standard interior penalty DG form:
\begin{eqnarray*}
a(\chi;p,q) &=& \sum_{E\in\mathcal{E}_h} (\chi\nabla p,\nabla q)_E
+ \sum_{e\in\Gamma_h\cup\Gamma_{p\D}} \sigma_p h_e^{-1} ([p],[q])_e \\
&& -\sum_{e\in\Gamma_h\cup\Gamma_{p\D}} (\{ \chi \nabla p\} \cdot\bfn_e, [q])_e
+\epsilon_p \sum_{e\in\Gamma_h\cup\Gamma_{p\D}} (\{ \chi \nabla q\}\cdot\bfn_e, [p])_e.
\end{eqnarray*}

The scalar $\epsilon_p$ is either equal to $-1$ or to $+1$ to yield a symmetric or non-symmetric bilinear form.
The penalty parameter $\sigma_p$ is a positive constant: it has to be sufficiently large if $\epsilon_p=-1$ \cite{Riviere2008}.
The discretization of the operator $-\mu\Delta \bfu - (\lambda+\mu) \nabla (\nabla \cdot\bfu)$ is also recalled:
\begin{eqnarray*}
c(\bfu,\bfv) &=& \mu \sum_{E\in\mathcal{E}_h} (\nabla \bfu,\nabla \bfv)_E
+ \mu \sum_{e\in\Gamma_h\cup\Gamma_{\bfu\D}} \sigma_\bfu h_e^{-1} ([\bfu],[\bfv])_e \\
&& -\mu \sum_{e\in\Gamma_h\cup\Gamma_{\bfu\D}} (\{  \nabla \bfu\} \bfn_e, [\bfv])_e
+\epsilon_\bfu \mu \sum_{e\in\Gamma_h\cup\Gamma_{\bfu\D}} (\{  \nabla \bfv\}\bfn_e, [\bfu])_e
\\
&& + (\lambda+\mu) \sum_{E\in\mathcal{E}_h} (\nabla\cdot\bfu,\nabla \cdot\bfv)_e
- (\lambda+\mu) \sum_{e\in\Gamma_h\cup\Gamma_{\bfu\D}} (\{\nabla \cdot \bfu\}, [\bfv\cdot\bfn_e])_e.
\end{eqnarray*}
The forms $\ell_w, \ell_o$ and $\ell_\bfu$ handle the source/sink functions, external forces and boundary conditions.
\begin{eqnarray*}
\ell_w(t_{n+1};q_h) &=& (f_w(t_{n+1}),q_h)_\Omega
+\epsilon_p \sum_{e\in\Gamma_{p\D}} (\lambda_w^n K \nabla q_h\cdot\bfn_e,p_{w\D}(t_{n+1}))_e
\\
&&+ \sum_{e\in\Gamma_{p\N}} (g_w(t_{n+1}),q_h)_e
+\sum_{e\in\Gamma_{p\D}} \sigma_p h_e^{-1} (p_{w\D}(t_{n+1}), q_h)_e,
\end{eqnarray*}
\begin{eqnarray*}
\ell_o(t_{n+1};q_h) &=& (f_o(t_{n+1}),q_h)_\Omega
+\epsilon_p \sum_{e\in\Gamma_{p\D}} (\lambda_o^n K \nabla q_h\cdot\bfn_e,p_{o\D}(t_{n+1}))_e
\\
&&+ \sum_{e\in\Gamma_{p\N}} (g_o(t_{n+1}),q_h)_e
+\sum_{e\in\Gamma_{p\D}} \sigma_p h_e^{-1} (p_{o\D}(t_{n+1}), q_h)_e,
\end{eqnarray*}
\begin{eqnarray*}
\ell_\bfu(t_{n+1};\bfv_h) &=& (\bff_\bfu(t_{n+1}),\bfv_h)_\Omega
+\epsilon_\bfu \mu \sum_{e\in\Gamma_{\bfu\D}} (\nabla \bfv_h \, \bfn_e, \bfu_\D(t_{n+1}))_e
\\
&&+ \sum_{e\in\Gamma_{\bfu\N}} (\bfg_\bfu(t_{n+1}),\bfv_h)_e
+\sum_{e\in\Gamma_{\bfu\D}} \sigma_\bfu h_e^{-1} (\bfu_\D(t_{n+1}),\bfv_h)_e.
\end{eqnarray*}
In order to start the algorithm, the solutions at times $t_0$ and $t_1$ are to be computed.
The initial values are chosen to be the $L^2$ projections of the initial data.
\[
(P_w^0,q_h)_\Omega = (p_w^0,q_h)_\Omega, \quad
(P_o^0,q_h)_\Omega = (p_o^0,q_h)_\Omega, \quad
(\bfU^0,\bfv_h)_\Omega = (\bfu^0,\bfv_h)_\Omega, \quad \forall q_h\in Q_h, \, \forall \bfv_h\in\bfV_h.
\]

To obtain $P_w^1$ we solve a modified flow equation:
\begin{equation}
(\mathcal{C}_1(P_o^0,P_w^0) \frac{P_w^{1}-P_w^0}{\tau_0}, q_h)_\Omega
+ a(\lambda_w^0 K; P_w^{1},q_h) 
 = \ell_w(t_1;q_h).\quad \forall q_h\in Q_h.\label{eq:initdiscpb1}
\end{equation}
Once $P_w^1$ is computed, we can solve for $P_w^0$ satisfying
\begin{align}
(\mathcal{C}_3(P_o^0,P_w^0) \frac{P_o^{1}-P_o^0}{\tau_0},q_h)_\Omega 
+ a(\lambda_o^0 K; P_o^{1},q_h) 
  = \ell_o(t_1;q_h)
- (\mathcal{C}_4(P_o^0,P_w^0) \frac{P_w^{1}-P_w^0}{\tau_0},q_h)_\Omega, \quad \forall q_h\in Q_h.
\label{eq:initdiscpb2}
\end{align}
Because $\tau_0$ is chosen to be much smaller than $\tau$, the consistency errors due to the modified equations \eqref{eq:initdiscpb1} and \eqref{eq:initdiscpb2} will be negligible
compared to the numerical errors for all time steps $n\geq 2$.
Finally, to compute the displacement $\bfU^1$, equation~\eqref{eq:discpb3} is used without the stabilization terms.  This yields a consistent discretization for the displacement at time step $t_1$.
\begin{align}
c(\bfU^{1},\bfv_h) = \ell_\bfu(t_{1};\bfv_h) -b_p(S_w^{1} P_w^{1}+(1-S_w^{1}) P_o^{1},\bfv_h)  
, \quad
\forall \bfv_h\in\bfV_h.  \label{eq:initdiscpb3}
\end{align}
Define the DG norm for discrete pressures:
\[
\Vert q_h \Vert_{\mathrm{DG}} = \left( \sum_{E\in\mathcal{E}_h} \Vert \nabla q_h\Vert_{L^2(E)}^2 
+ \sum_{e\in\Gamma_h\cup\Gamma_{p\D}} h_e^{-1} \Vert [q_h]\Vert_{L^2(e)}^2 \right)^{1/2}, \quad \forall q_h\in Q_h.
\]
A similar norm is defined for vector-valued functions $\bfv_h\in\bfV_h$; it differs by the boundary terms.
\[
\Vert \bfv_h \Vert_{\mathrm{DG}} = \left( \sum_{E\in\mathcal{E}_h} \Vert \nabla \bfv_h\Vert_{L^2(E)}^2 
+ \sum_{e\in\Gamma_h\cup\Gamma_{\bfu\D}} h_e^{-1} \Vert [\bfv_h]\Vert_{L^2(e)}^2 \right)^{1/2}, \quad \forall \bfv_h\in \bfV_h.
\]
We now recall the coercivity properties for the bilinear forms $a$ and $c$.  
\begin{lemma}\label{lem:coerc}
Let $\chi$ be a scalar-valued function bounded below and above by  positive constants $C_{\underline{\chi}}$ and $C_{\overline{\chi}}$. If $\epsilon_p = -1$, assume that $\sigma_p$ is sufficiently large. The following holds:
\begin{equation}\label{eq:aPcoer}
\frac12 \Vert q_h \Vert_{\mathrm{DG}}^2 \leq a(\chi;q_h,q_h), \quad \forall q_h\in Q_h.
\end{equation}
In addition, assume the penalty parameter $\sigma_{\bfu}$ is sufficiently large. Then we have
\begin{equation}\label{eq:cUcoer}
\frac12 \Vert \bfv_h \Vert_{\DG}^2 \leq c(\bfv_h,\bfv_h), \quad \forall \bfv_h\in\bfV_h.
\end{equation}
\end{lemma}
The proof of Lemma~\ref{lem:coerc} is classical and is therefore skipped \cite{Riviere2008}. If $\epsilon_p = -1$, the constant $\sigma_p$ depends on trace constants and the constants $C_{\underline{\chi}}$ and $C_{\overline{\chi}}$.  Similarly, the penalty parameter $\sigma_{\bfu}$ depends on trace constants and on the Lam\'e parameters.

Next we show that the discrete equations are solvable under some conditions on the phase mobilities.
\begin{proposition}
Assume that the functions $\lambda_w$ and $\lambda_o$ are bounded below by positive constants.
For any $n\geq 0$, the solutions $(P_w^n, P_o^n, \bfU^n)$ exist and are unique. 
\end{proposition}
\begin{proof}
Existence and uniqueness of the initial solutions $(P_w^0, P_o^0, \bfU^0)$ is immediate  because of the $L^2$ projection operator.
Regarding the solutions at time $t_1$, since \eqref{eq:initdiscpb1}, \eqref{eq:initdiscpb2}, \eqref{eq:initdiscpb3} are linear problems in finite dimension, it suffices to show uniqueness.  The proof is an immediate consequence of the coercivity Lemma~\ref{lem:coerc} and the
non-negative signs of the coefficients $\mathcal{C}_1$ and $\mathcal{C}_3$ (see \eqref{eq:nonneg}). Next we prove existence of solutions
to \eqref{eq:discpb1}-\eqref{eq:discpb3} by also utilizing the fact that these equations are linear with respect to their unknowns. It is thus equivalent to show uniqueness.  Fix $n\geq 1$ and assume that $\tilde{P}_w$ is the difference of two solutions to \eqref{eq:discpb1}. We have
\[
(\mathcal{C}_1(P_o^n,P_w^n) \frac{\tilde{P}_w}{\tau}, q_h)_\Omega
+ a(\lambda_w^n K; \tilde{P}_w,q_h) = 0, \quad \forall q_h\in Q_h.
\]
Choosing $q_h = \tilde{P}_w$ in the equation above and using \eqref{eq:aPcoer} and \eqref{eq:nonneg}, we have that $\tilde{P}_w=0$.
Next, we denote by $\tilde{P}_o$ the difference of two solutions to \eqref{eq:discpb2}; it satisfies
\[
(\mathcal{C}_3(P_o^n,P_w^n) \frac{\tilde{P}_o}{\tau},q_h )_\Omega
+ a(\lambda_o^n K; \tilde{P}_o,q_h)   = 0, \quad \forall q_h \in Q_h.
\]
Again, by choosing $q_h = \tilde{P}_o$ and using \eqref{eq:aPcoer} and \eqref{eq:nonneg}, we have that $\tilde{P}_o=0$.
Finally, let $\tilde{\bfU}$ be the difference of two solutions to \eqref{eq:discpb3}. It satisfies
\[
c(\tilde{\bfU},\bfv_h) 
+\gamma \left(\frac{\tilde{\bfU}}{\tau },\bfv\right)_\Omega
=0, \quad \forall \bfv_h\in\bfV_h. 
\]
Choosing $\bfv_h = \tilde{\bfU}$ and using \eqref{eq:cUcoer}, yields
\[
\frac12 \Vert \tilde{\bfU}\Vert_{\DG}^2 + \gamma \Vert \tilde{\bfU}\Vert_{L^2(\Omega)}^2 = 0,
\]
which gives the desired result.
\end{proof}

\section{Numerical Results}
\label{sec:numer}

We first verify the optimal rate of convergence of our proposed numerical method for smooth solutions and then  we apply our scheme
to various porous media problems: the McWorther problem, a non-homogeneous medium with different capillary pressures, 
a medium subjected to load, and a medium with highly varying permeability and porosity.
Unless explicitely stated in the text, all examples use the following physical parameters.
\begin{align*}
\mu_w=\mu_o = 0.001 \, \mbox{Pa s}, \, K_w=K_{o}= 10^{10} \, \mbox{Pa}, \\ \lambda = 7142857 \,  \mbox{Pa}, \,  \mu=1785714  \, \mbox{Pa},  \, K_s=8333333 \, \mbox{Pa}, \\
\phi=0.3,  \,\alpha=0.8,  \, \epsilon_p = \epsilon_\bfu = -1. 
\end{align*}

The linear systems are solved by LU preconditioned GMRES with absolute stopping criteria $10^{-12}$.  Most of the problems converged with desired accuracy in 1 or 2 iterations.


\subsection{Convergence Rates}

We employ the method of manufactured solutions to test the convergence rates of our scheme.  The exact solution is smooth and defined
by
\[
p_w(x,y,z) = \sin(y)+5, \quad p_o(x,y,z)=\cos(x)+25, \quad \bfu(x,y,z) = (\cos(x),\sin(y),\cos(z+x))^T.
\]
The following physical parameters are chosen: $\phi = 0.3, K=1, \lambda=1,\mu=0.6, K_w=K_o=K_s=10, \alpha=0.9, \lambda_w(s_w) = s_w, \lambda_o(s_w)=1-s_w$ and $p_d =10$.
The computational parameters are $\tau = 1, \tau_0=10^{-2}, \sigma_p = 20, \sigma_\bfu=14$ and $\gamma=10$.
The domain is the unit cube partitioned into tetrahedra. No cut-off operator is applied in this example. 
We compute the numerical errors at the final time $T=5$ on a series of uniformly refined meshes.
\[
e_{w} = p_w(T)-P_w^N, \quad e_o = p_o(T)-P_o^N, \quad e_\bfu = \bfu(T)-\bfU^N.
\]
Table~\ref{tab:errorpressure} displays the errors for the phase pressures in the
broken gradient norm and the $L^2$ norm, and the errors for the displacement in the $L^2$ norm. The rates are optimal.

\begin{table}[H]
\centering
\vspace{-0.5em}
\begin{tabular}{|c|c|c|c|c|c|c|c|c|c|c|}
\hline
$h$  & $\Vert e_w\Vert$ & Rate & $\Vert \nabla_h e_w \Vert $& Rate
& $\Vert e_o\Vert$ & Rate & $\Vert \nabla_h e_o \Vert $& Rate & $\Vert e_\bfu \Vert$ & Rate\\
\hline
1/2    & 5.78e-03   &  & 6.89e-02  &  &  7.53e-03   &  & 1.08e-01 &    & 1.16e-02 & \\
1/4   &  1.56e-03   & 1.89 &   3.57e-02  & 0.95 & 2.01e-03  & 1.91 &   5.48e-02   & 0.98 &  3.03e-03  & 1.93\\
1/8   & 4.03e-04   & 1.95 & 1.80e-02   & 0.99 &  5.24e-04  & 1.94 & 2.75e-02  & 0.99 &  7.79e-03    & 1.94\\
\hline
\end{tabular}
\caption{Numerical errors and rates for the numerical approximations of smooth solutions.}
\label{tab:errorpressure}
\end{table}

\subsection{McWhorter Problem}

The original McWhorter problem simulates counter-current flow in a homogeneous one-dimensional domain. 
Because of the quasi-analytical solution developed in 
\cite{McwhortherSunada1990}, this benchmark problem is ideal for evaluating the accuracy of a numerical scheme.
The fluid phases are incompressible, which means that the inverse of the bulk modulus for each phase is set to zero.  
The entry pressure (see \eqref{eq:pc}) is $p_d=5000$ Pa. 
For this problem, the Biot-Willis constant is set equal to $1$ and the permeability is $K=10^{-10}$. 
We solve this problem in a thin slab $[0,2.6]\times [0,0.065]\times [0,0.0325]$ partitioned into 160 cubes of side $h=0.0325$, each
cube is then divided into 6 tetrahedra. 
The computational parameters are: 
\[
\tau=1	\mbox{ s}, \quad \tau_0 = 0.01 \mbox{ s}, \quad \sigma_p = 400, \quad \sigma_\bfu = 1000, \quad \gamma = 10^5, \quad T=5000 \mbox{ s}.
\]
Initially, the pressures are $p_w^0 = 184000$ Pa and $p_o^0 = 234000$ Pa,  which implies the initial saturation in the domain is $s_w^0 = 0.01$. \Bk The Dirichlet boundary is the left vertical boundary $\{0\}\times [0,0.065]\times [0,0.0325]$. Dirichlet data are selected such that the wetting phase saturation is equal to $0.99$ on that boundary.  This means that
$p_{w\D} = 194970$  Pa and $p_{o\D} = 200000$ Pa.
No flow is imposed on the remainder of the boundary: $g_w = g_o = 0$.
Zero displacement is prescribed on both left and right vertical boundaries and no traction ($ \bfg_\bfu = {\bf 0}$) is prescribed on the remainder of the boundary.
\[
\bfu_\D = {\bf 0} \quad \mbox{on} \quad \{0\}\times [0,0.065]\times [0,0.0325] \cup \{2.6\} \times [0,0.065]\times [0,0.0325].
\]
The saturation profiles at different times are plotted in Fig.~\ref{fig:mcwhorter_sat_profile_smin001h00325}. We observe that the numerical solution coincide with the analytical solution. 
\begin{figure}[H]
\centering
\includegraphics[width=0.65\linewidth]{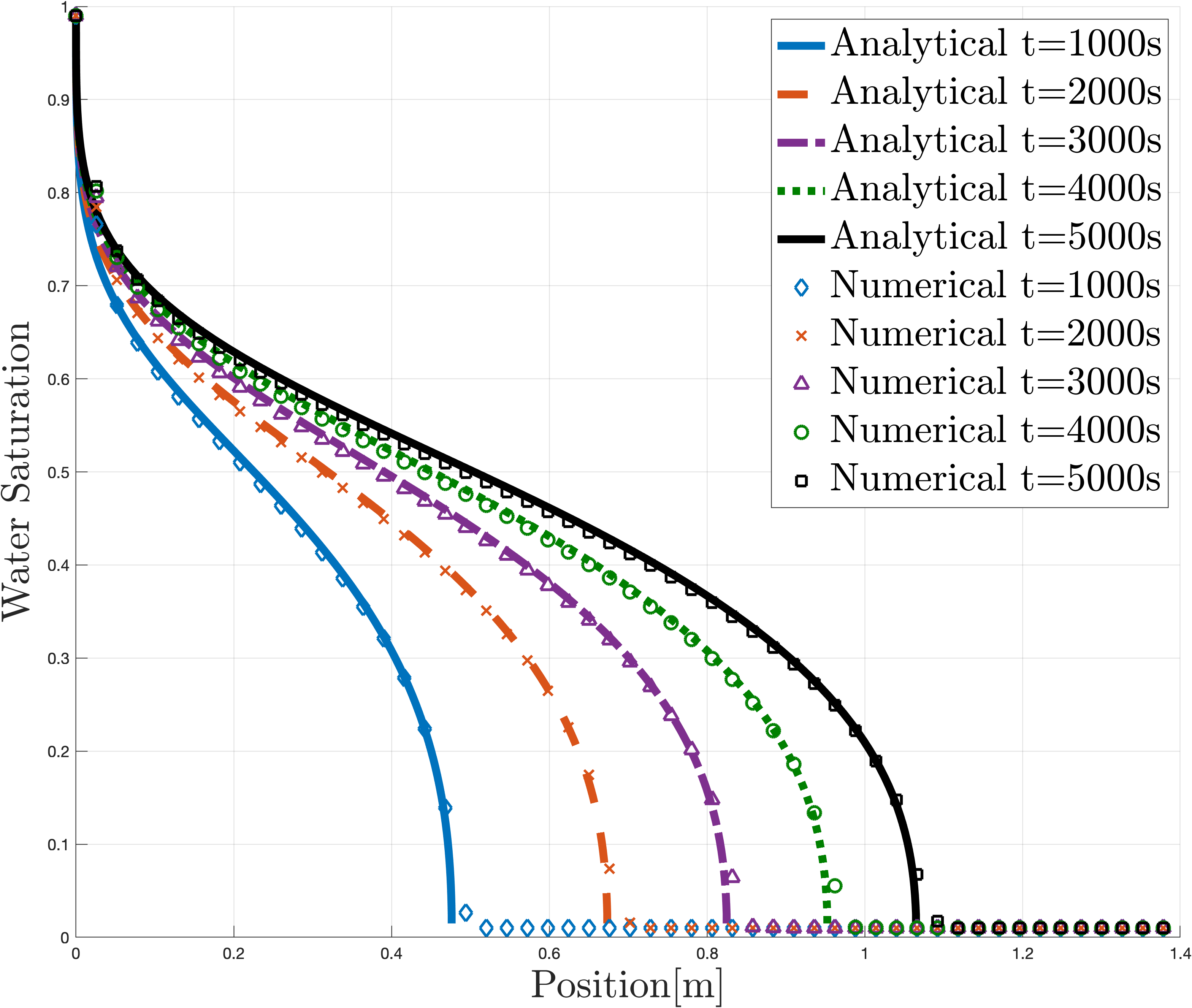}
\caption{McWhorter problem: wetting phase saturation profiles at five selected time steps.}
\label{fig:mcwhorter_sat_profile_smin001h00325}
\end{figure}
%
In Fig.~\ref{fig:mcwhorter_displacement}, we compare the numerical displacement obtained with our method with the numerical displacement
obtained by a finite volume discretization in \cite{asadi2015comparison} at $t=1000$s.  Because there are no external forces, changes in the displacement are caused by changes in the pressures. We observe a good agreement between the two solutions.
\begin{figure}[H]
\centering
\includegraphics[width=0.45\linewidth]{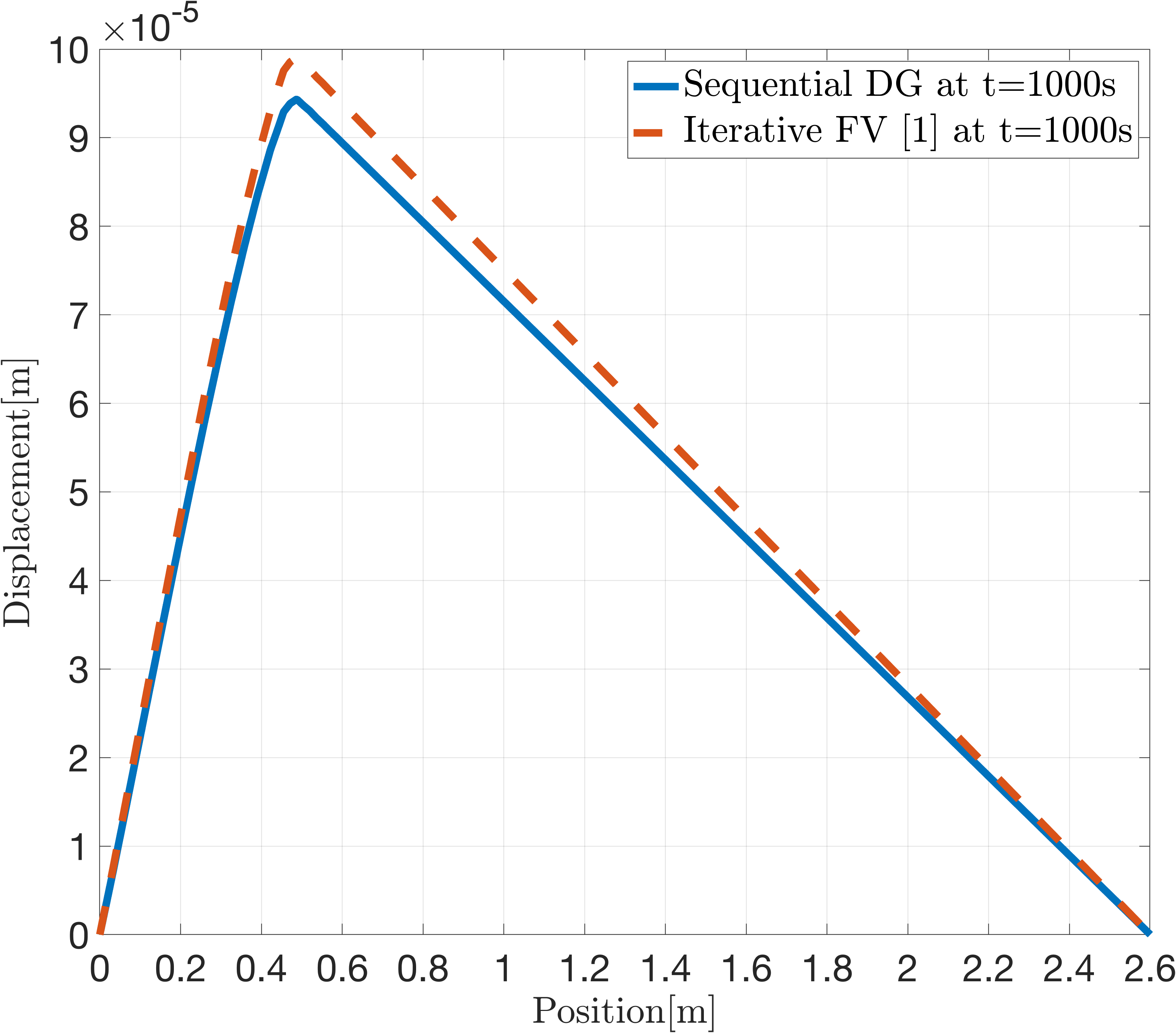}
\caption{McWhorter problem: displacement at $t=1000$s.}
\label{fig:mcwhorter_displacement}
\end{figure}

\subsection{Porous Medium with Heterogeneous Inclusions}
This example considers a porous medium with two rock types with different permeability and
entry pressure in each rock.  The domain $\Omega = [0,100]\times [0,100]\times [0,2.5]$ (m$^3$) contains
two box-shape inclusions $[20,40]\times[50,70]\times[0,2.5]$ (m$^3$) and $[50,90]\times[20,50]\times[0,2.5]$ (m$^3$) (see Fig.~\ref{fig:2blocksetup}).  The permeability and entry pressure for rock type 1 (resp. type 2) are denoted by $K_1$ and $p_{d1}$ (resp. $K_2$ and $p_{d2}$). 
We consider two cases:
\begin{align*}
\mbox{Case 1:} & \quad K_1 = 4.2\times 10^{-11}, \, p_{d1} =\sqrt{2}p_{d2},  \, K_2 = 2 K_1, \, p_{d2} = 5000,\\
\mbox{Case 2:} & \quad K_1 = 8.4\times 10^{-11}, \, p_{d1} = 5000, \, K_2 = K_1/2, \, p_{d2} =\sqrt{2}p_{d1}.
\end{align*}

\vspace{-1em}

\begin{figure}[H]
\subfigure[top view \label{fig:2blocksetuptop}]{
\includegraphics[width=0.32\linewidth]{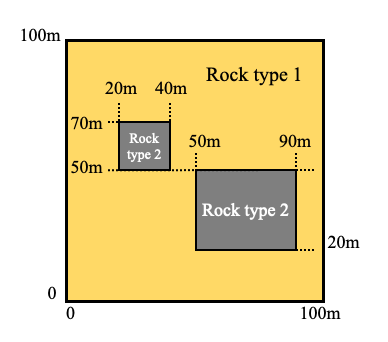}}
\subfigure[flow BCs \label{fig:2blockBCflow}]{
\includegraphics[width=0.32\linewidth]{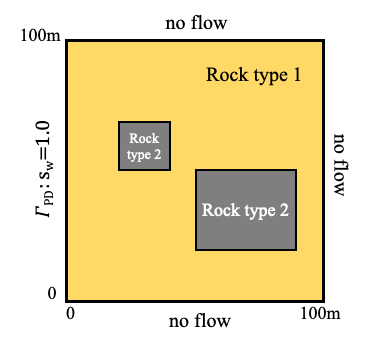}}
\subfigure[geomechanics BCs \label{fig:2blockBCdisp}]{
\includegraphics[width=0.32\linewidth]{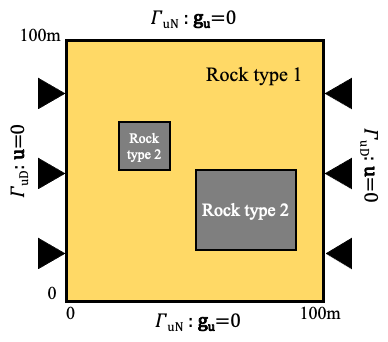}}
\caption{Domain with two inclusions: top view and set-up of boundary conditions for flow and geomechanics. }
\label{fig:2blocksetup}
\end{figure}

The initial non-wetting phase pressure is $p_o^0 = 200000$ Pa and the initial wetting phase pressure is chosen so that the initial wetting phase saturation in the areas of rock type 1 and rock type 2 are 0.1 and 0.05 respectively. Dirichlet data are selected such that the wetting phase saturation is equal to $1.0$ on the left side $\{0\}\times [0,100]\times [0,2.5]$, 
  this means that $p_{w\D} = 195000$  Pa and $p_{o\D} = 200000$ Pa on that side.
No flow is imposed on the remainder of the boundary: $g_w = g_o = 0$.
Zero displacement is prescribed on both left and right sides and no traction ($ \bfg_\bfu = {\bf 0}$) is prescribed on the remainder of the boundary.
The domain is partitioned into $9600$ tetrahedra. The computational parameters are:
\begin{equation}
\tau=5 \mbox{ days}, \quad \tau_0 = 0.05\mbox{ days}, \quad \sigma_p = 800, \quad \sigma_\bfu = 800, \quad \gamma = 10^5,\quad T=1000 \mbox{ days}.
\label{eq:comppar}
\end{equation}
First, we simulate flow for Case 1. 
Fig.~\ref{fig:blocks_case1_sat} shows the wetting phase saturation contours at 50, 125, 250, 375, 500 and 1000 days.  
The saturation front avoids the inclusions that have lower permeability, as expected. As the wetting phase floods the medium, deformations occur; for better visualization the displacement components are scaled by $1200$.

\begin{figure}[H]
\vspace{-0.5em}
\centering     
\subfigure[$t=50$]{\includegraphics[width=0.29\linewidth]{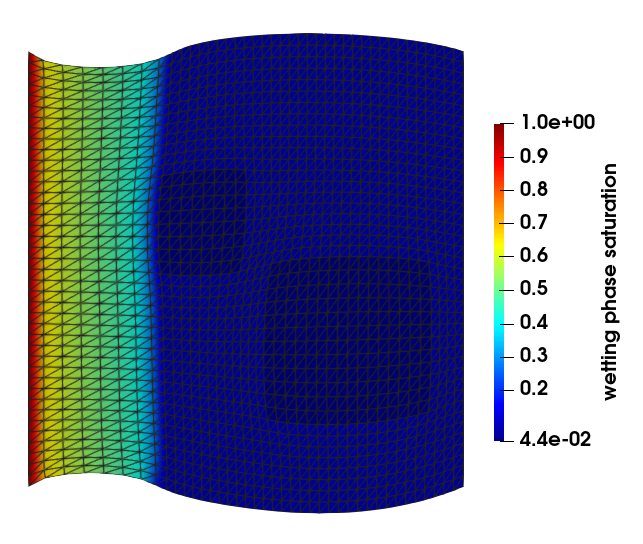}}
\subfigure[$t=125$]{\includegraphics[width=0.29\linewidth]{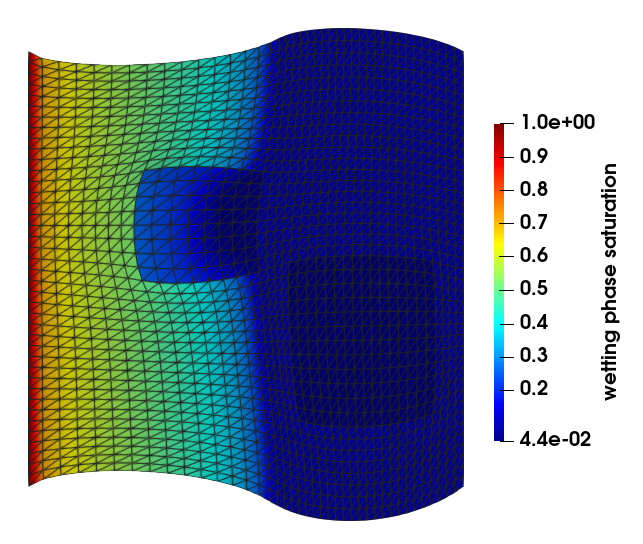}}
\subfigure[$t=250$]{\includegraphics[width=0.29\linewidth]{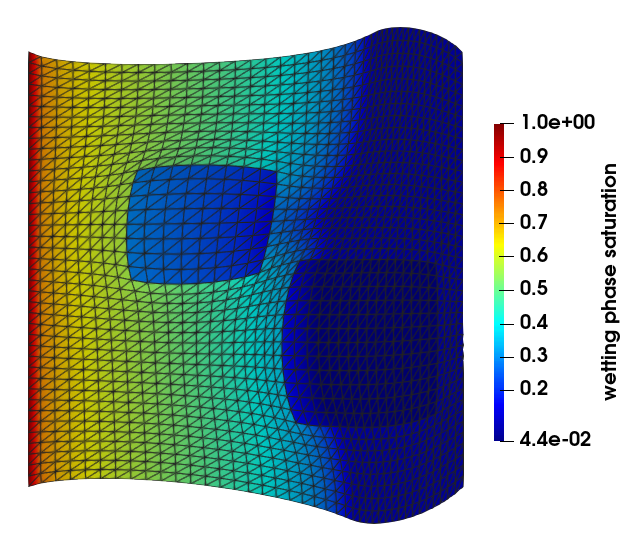}}\\
\vspace{-0.5em}
\subfigure[$t=375$]{\includegraphics[width=0.29\linewidth]{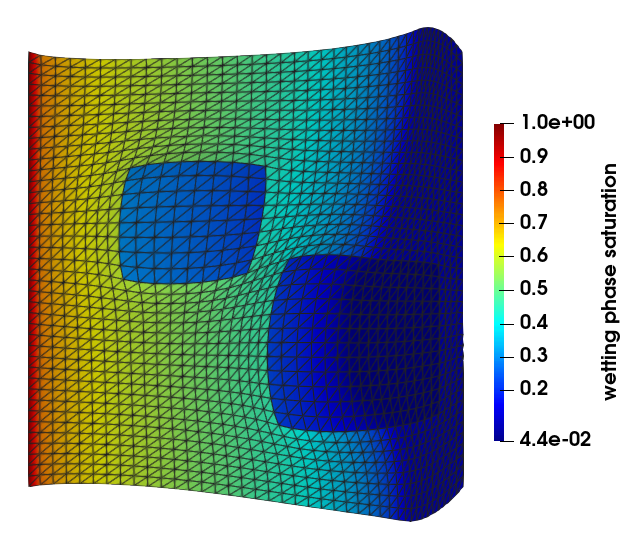}}
\subfigure[$t=500$]{\includegraphics[width=0.29\linewidth]{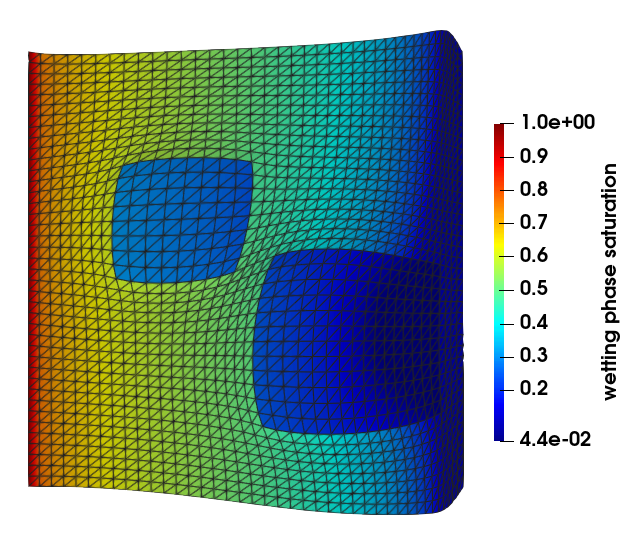}}
\subfigure[$t=1000$]{\includegraphics[width=0.29\linewidth]{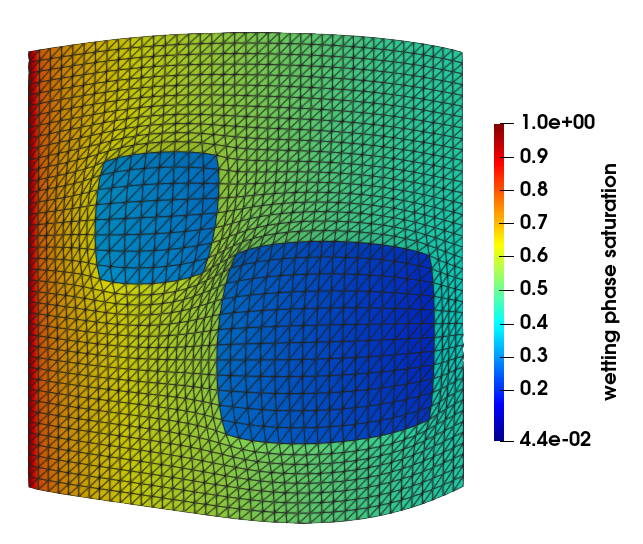}}
\caption{Heterogeneous inclusions problem for Case 1:  wetting phase saturation contours at  $t=50, 125,  250,  375,  500$  and $1000$ days.}
\label{fig:blocks_case1_sat}
\end{figure}
Profiles of the saturation front are plotted along two horizontal lines $y=35$m and $y=60$m in the plane $z=2.5$m for different times in Fig.~\ref{fig:block_case1_sat_linechart}. 
We observe that the saturation is discontinuous at the interface between the two types of rocks. The discontinuity is due to the capillary pressure function that switches to another curve as shown in Figure \ref{fig:blocks_case1_cap_pres}.  This is attributed to the fact that the entry pressures are discontinuous, the entry pressure in rock of type $2$ is smaller than the entry pressure in rock of type $1$.  We note that the threshold saturation $S_w^\ast \approx 0.84$, which is defined as $p_{c1}(S_w^\ast)=p_{c2}(1)$, is larger than the saturation  in rock 2, $S_{w2}$, and less than the saturation  in rock 1, $S_{w1}$,  therefore the phase pressure is continuous across the interface.    Figure \ref{fig:blocks_case1_pres} shows the wetting phase pressure solutions at different times. The inclusions impact the pressure contours: even though the permeability in rock 2 is twice the permeability in rock 1, the wetting phase saturation is smaller in rock 2, which yields a smaller wetting phase relative permeability. 


\begin{figure}[H]
\centering
\subfigure{\includegraphics[width=0.45\linewidth]{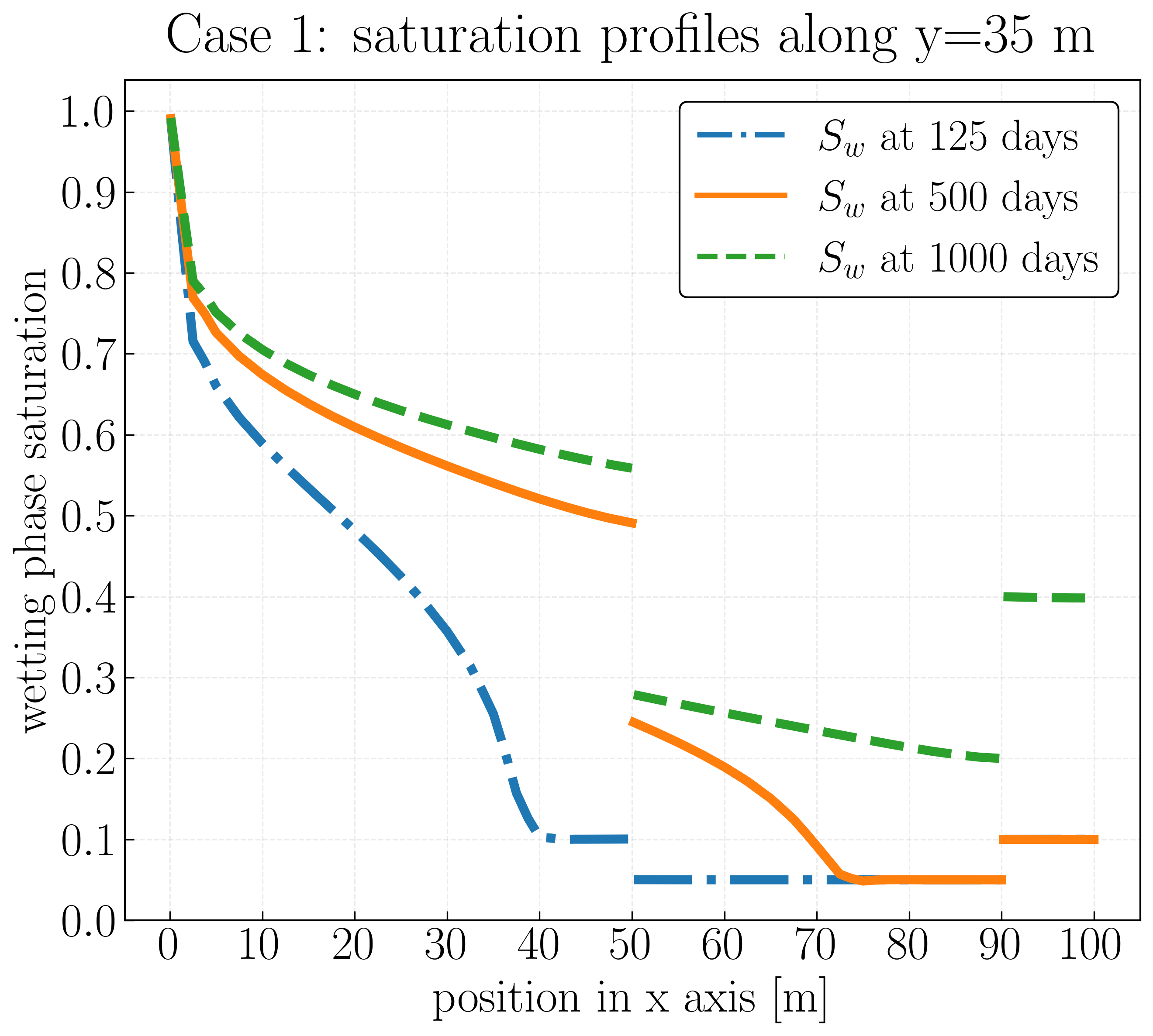}}
\subfigure{\includegraphics[width=0.45\linewidth]{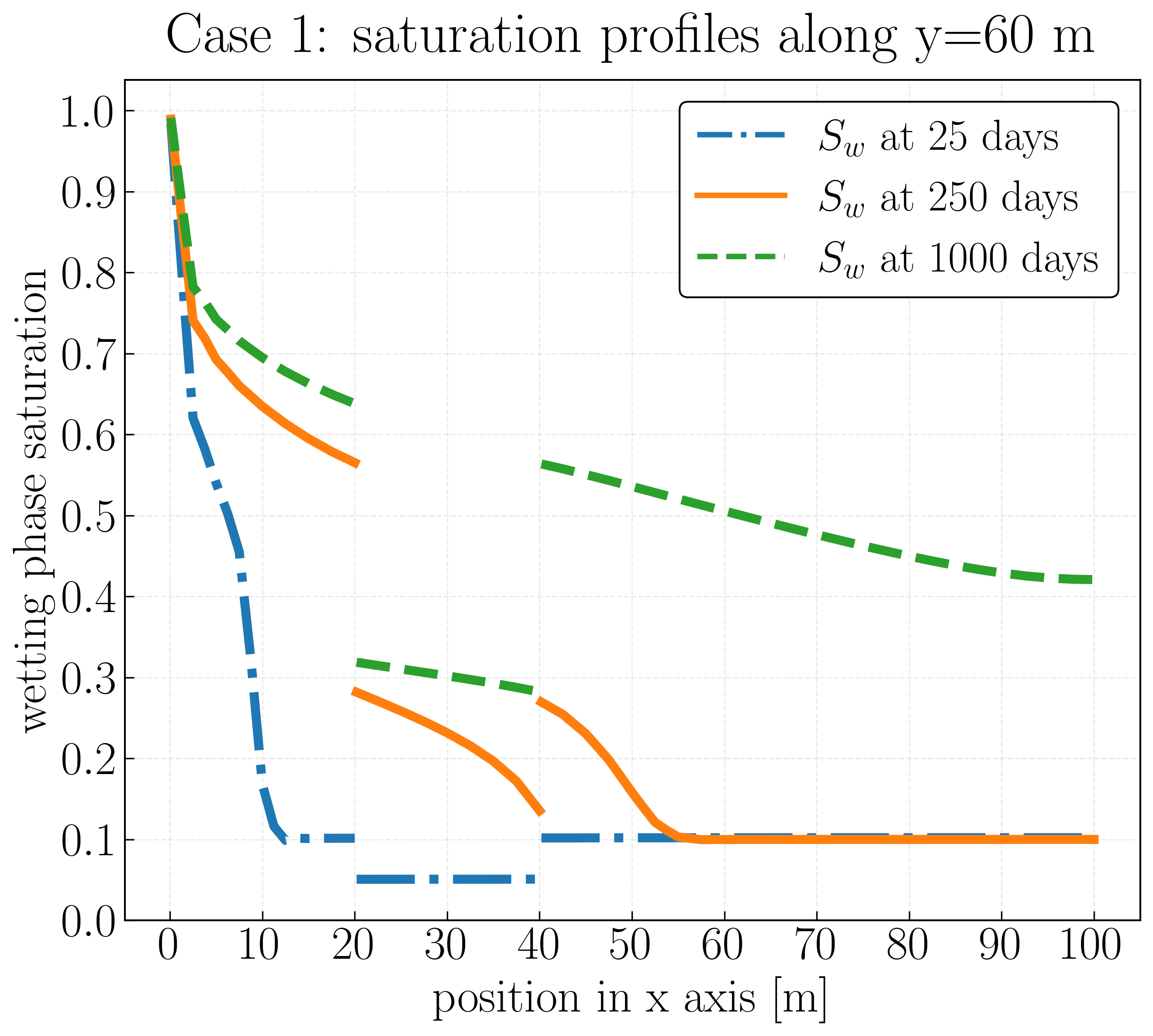}}
\caption{Heterogeneous inclusions problem for Case 1:  wetting phase saturation profiles along $y=35$ m (left) and $y=60$m (right) at selected times.}
\label{fig:block_case1_sat_linechart}
\end{figure}

\begin{figure}[H]
\centering
\includegraphics[width=0.5\linewidth]{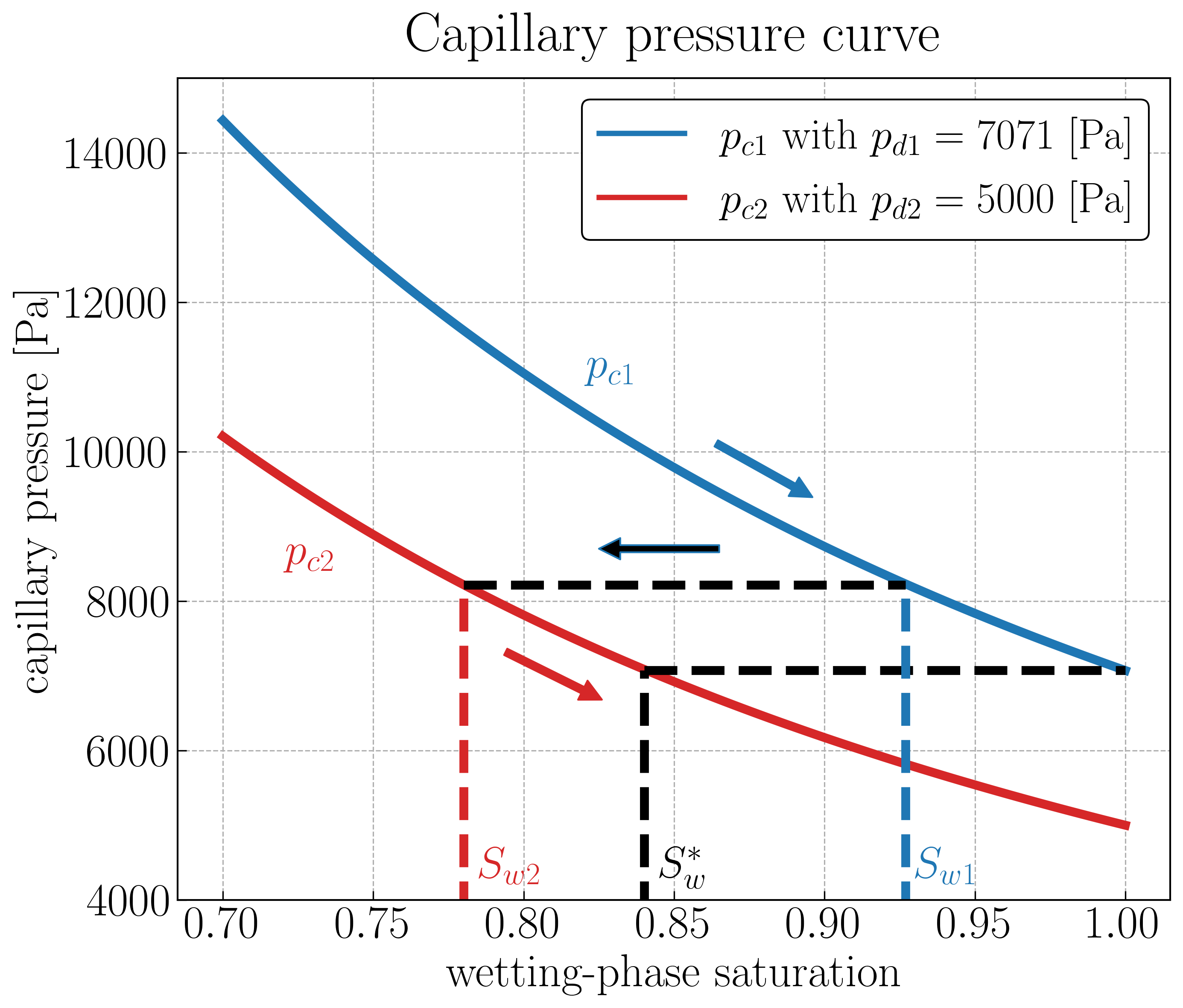}
\caption{Heterogeneous inclusions problem for Case 1: capillary pressure functions for the two rocks. } 
\label{fig:blocks_case1_cap_pres} 
\end{figure}

Since rock type 2 has a lower entry pressure,  less non-wetting phase is displaced by the wetting phase and the wetting phase saturation value lags behind in the region of rock type 2.  Overall,  the magnitude of displacement in the area of rock type 2 is smaller than in surrounding areas.  Figure \ref{fig:blocks_case1_displacement} shows the magnitude of the displacement at different times. 

Before the wetting phase front reaches the right boundary,  we first observe a significant displacement in the x-axis direction compared to the y- and z-axis directions.  More wetting phase passes through the area of rock type 1 where the medium is being stretched in the x-axis positive direction along with the flow.  Meanwhile,  the displacements in both the y-axis and z-axis increase in the direction that is perpendicular to the flow's direction.  This can be identified when the medium contracts in the y-axis when the wetting phase entered the domain. The same phenomenon can be observed when the region between two blocks is being stretched. The area that is close to the right boundary is being squeezed in the x-axis which leads to the increase of displacement in y- and z-direction until it bounces back due to the zero displacement boundary condition on the right side.  

\begin{figure}[H]
\centering     
\subfigure[$t=50$]{\includegraphics[width=0.29\linewidth]{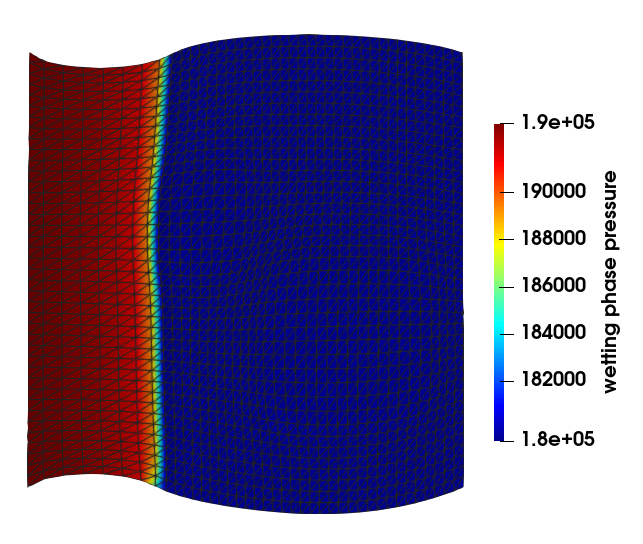}}
\subfigure[$t=125$]{\includegraphics[width=0.29\linewidth]{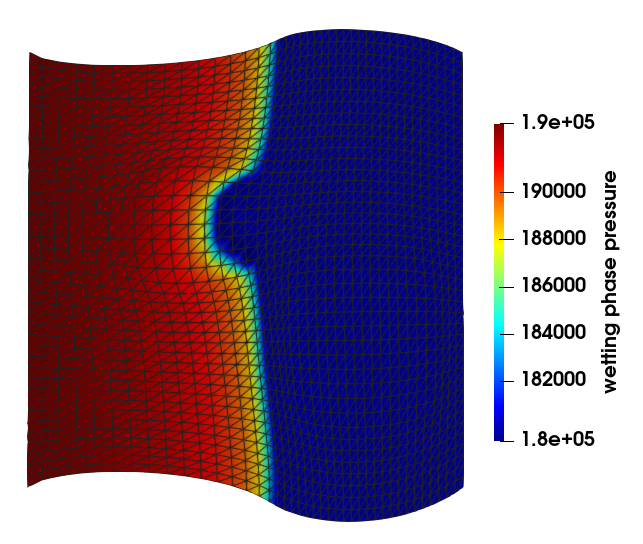}}
\subfigure[$t=250$]{\includegraphics[width=0.29\linewidth]{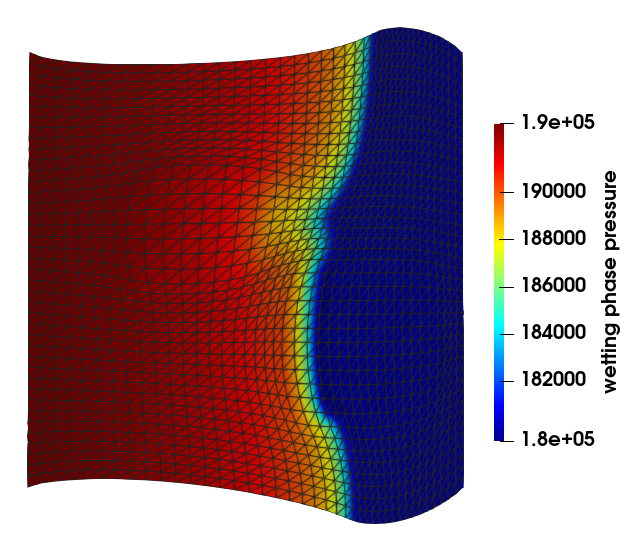}}\\
\subfigure[$t=375$]{\includegraphics[width=0.29\linewidth]{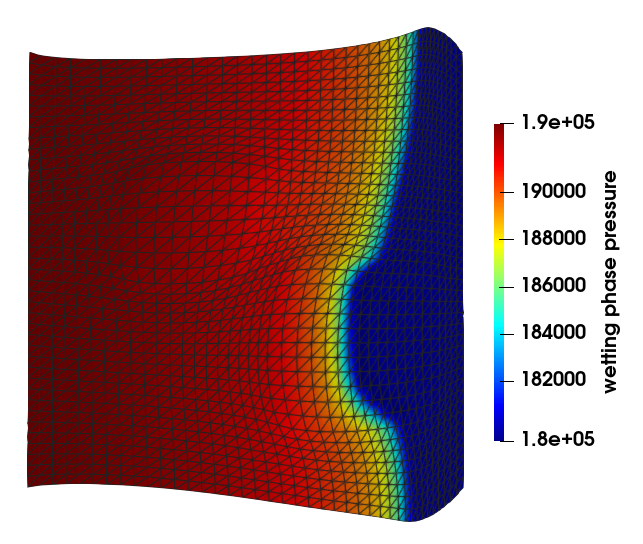}}
\subfigure[$t=500$]{\includegraphics[width=0.29\linewidth]{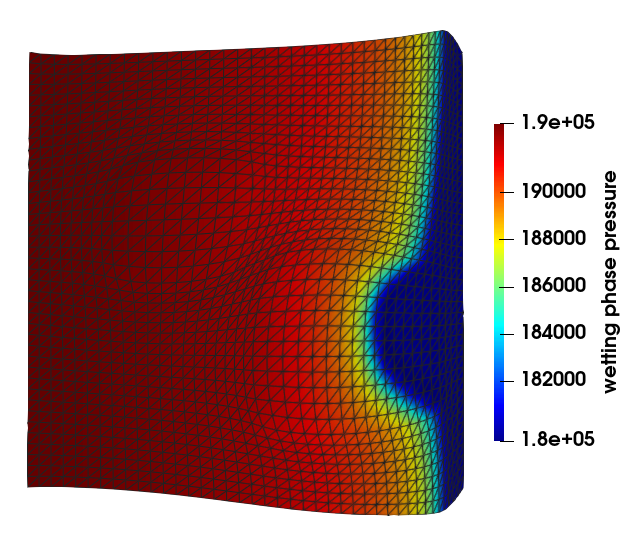}}
\subfigure[$t=1000$]{\includegraphics[width=0.29\linewidth]{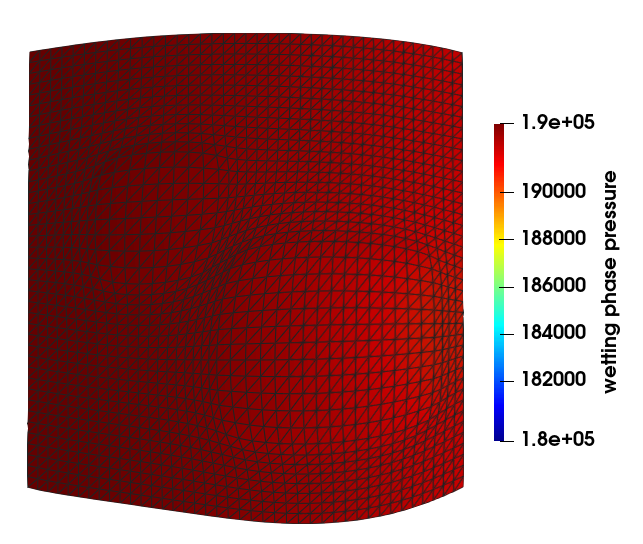}}
\caption{Heterogeneous inclusions problem for Case 1: wetting phase pressure contours at  $t=50, 125,  250,  375,  500$  and $1000$ days.}
\label{fig:blocks_case1_pres}
\end{figure}
\begin{figure}[H]
\centering     
\subfigure[$t=50$]{\includegraphics[width=0.29\linewidth]{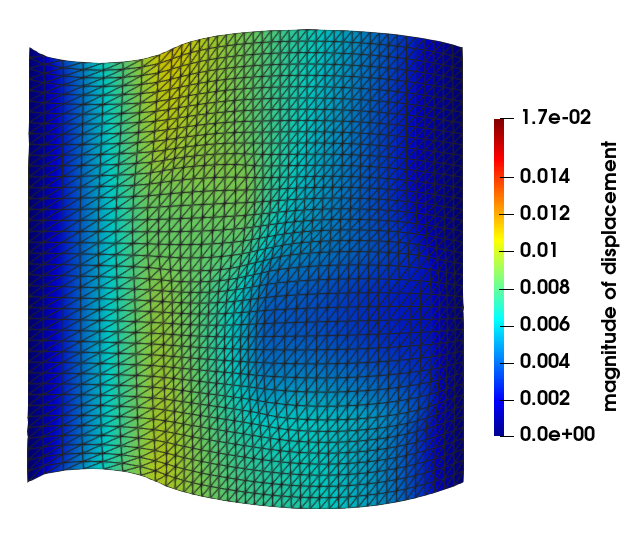}}
\subfigure[$t=125$]{\includegraphics[width=0.29\linewidth]{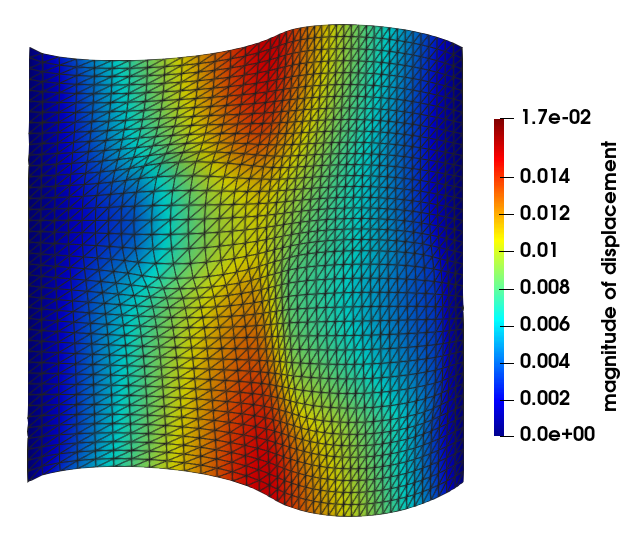}}
\subfigure[$t=250$]{\includegraphics[width=0.29\linewidth]{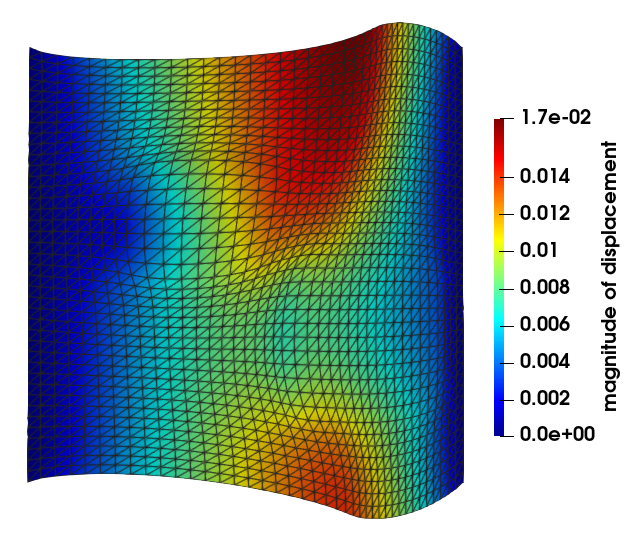}}\\
\subfigure[$t=375$]{\includegraphics[width=0.29\linewidth]{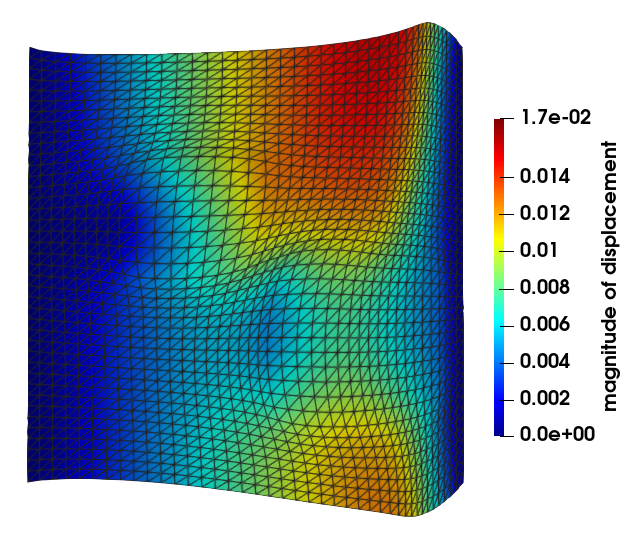}}
\subfigure[$t=500$]{\includegraphics[width=0.29\linewidth]{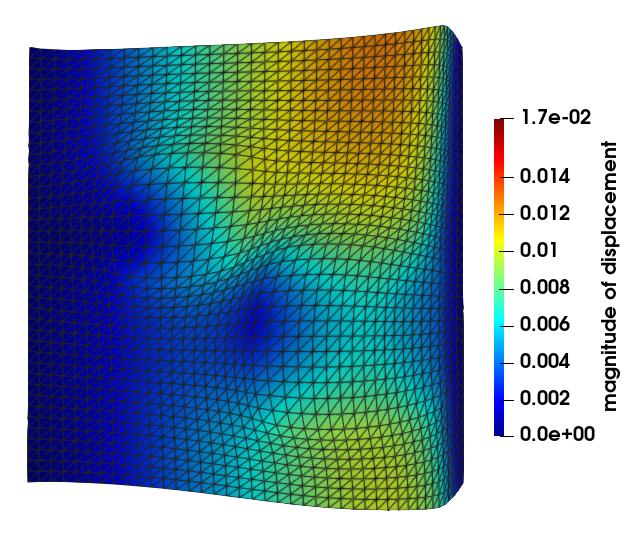}}
\subfigure[$t=1000$]{\includegraphics[width=0.29\linewidth]{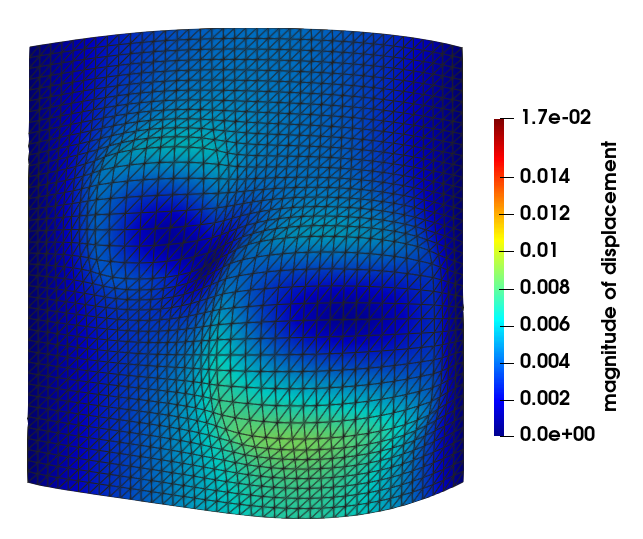}}
\caption{Heterogeneous inclusions problem for Case 1: magnitude of displacement  at $t=50, 125,  250,  375,  500$  and $1000$ days.}
\label{fig:blocks_case1_displacement}
\end{figure}

In the next experiments, we consider Case 2 where the rock properties are switched compared to Case 1.  
Initially, the wetting and non-wetting phase pressures are constant ($p_o^0=200000$ Pa) and the initial wetting phase saturation
in the areas of rock type 1 and rock type 2 are 0.1 and 0.2 respectively.
The saturation contours and profiles are shown in 
Fig.~\ref{fig:blocks_case2_sat} 
and Fig.~\ref{fig:block_case2_sat_linechart}  respectively. 
Since the saturation in the area of rock type 1, $S_{w1}$,  is less than threshold saturation $S^*_{w}$ (see Fig.~\ref{fig:blocks_case2_cap_pres}), the phase pressure is continuous across the interface.
Wetting phase pressure and magnitude of displacement are
presented in Fig.~\ref{fig:blocks_case2_pres}  and Fig.~\ref{fig:blocks_case2_displacement} respectively.

As seen in Figure \ref{fig:blocks_case2_pres},  the wetting phase pressure propagates in the area of rock type 2 faster than in the area of rock type 1 due to higher initial wetting phase saturation.  Higher wetting phase saturation indicates that there is more wetting phase that goes into the rock type 2 region (see Fig.~\ref{fig:blocks_case2_sat}).  This leads to a significant displacement of the rock type 2 in the x-axis and y-axis directions.  
Finally we remark that in the z-direction, the regions of rock 2 contract for Case 1 whereas they expand for Case 2 
(see Fig.~\ref{fig:3dzdirection}).

\begin{figure}[H]
\centering     
\subfigure[$t=50$]{\includegraphics[width=0.29\linewidth]{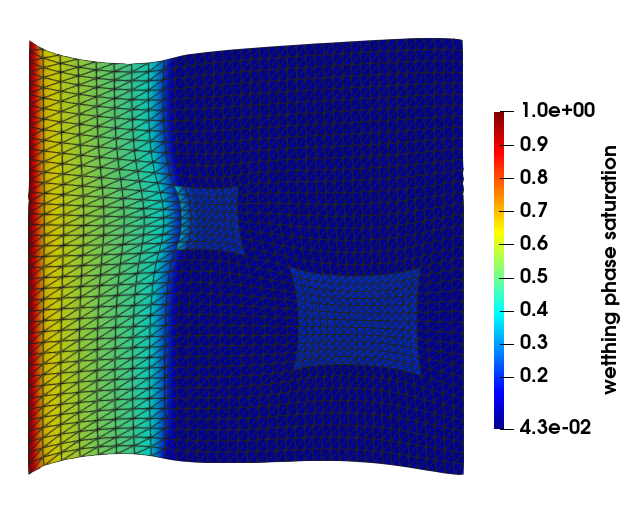}}
\subfigure[$t=125$]{\includegraphics[width=0.29\linewidth]{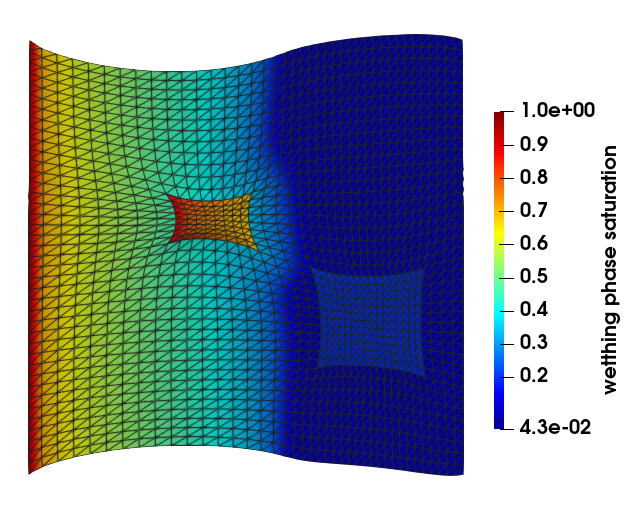}}
\subfigure[$t=250$]{\includegraphics[width=0.29\linewidth]{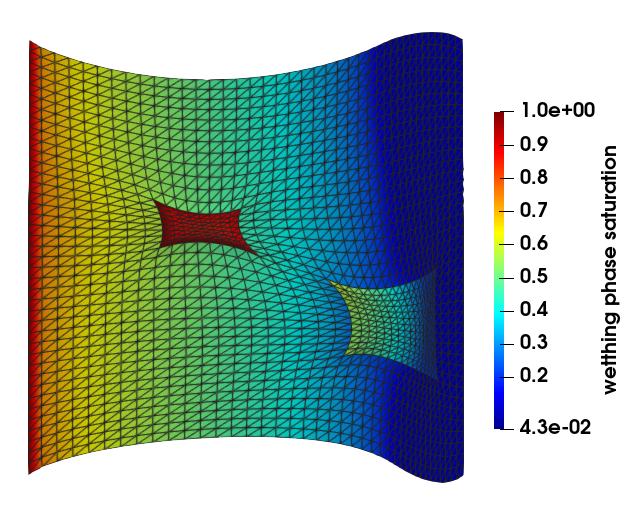}}\\
\subfigure[$t=375$]{\includegraphics[width=0.29\linewidth]{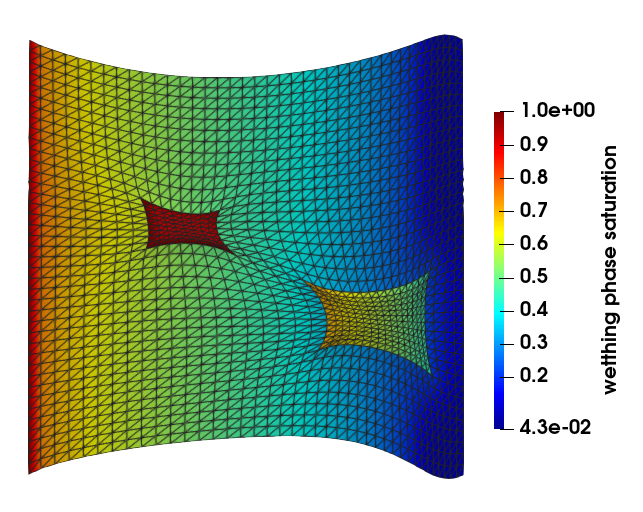}}
\subfigure[$t=500$]{\includegraphics[width=0.29\linewidth]{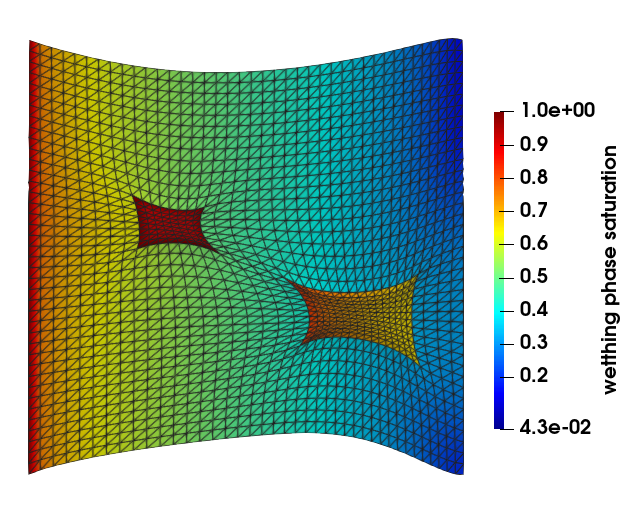}}
\subfigure[$t=1000$]{\includegraphics[width=0.29\linewidth]{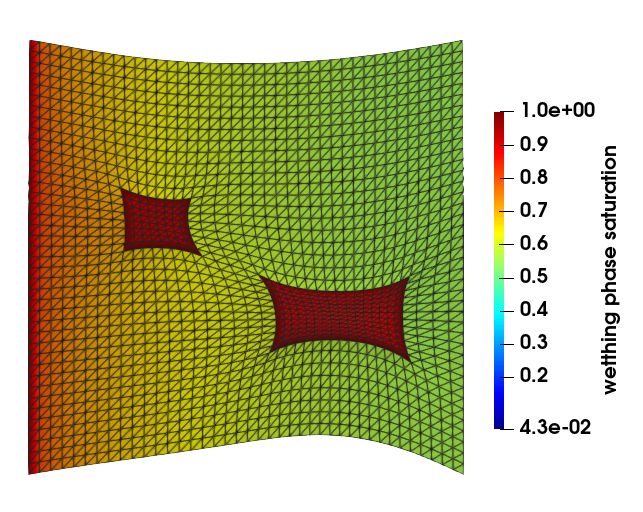}}
\caption{Heterogeneous inclusions  problem with switched rock types: wetting phase saturation contours at $t=50, 125,  250,  375,  500$  and $1000$ days.}
\label{fig:blocks_case2_sat}
\end{figure}

\begin{figure}[H]
\centering
\subfigure{\includegraphics[width=0.45\linewidth]{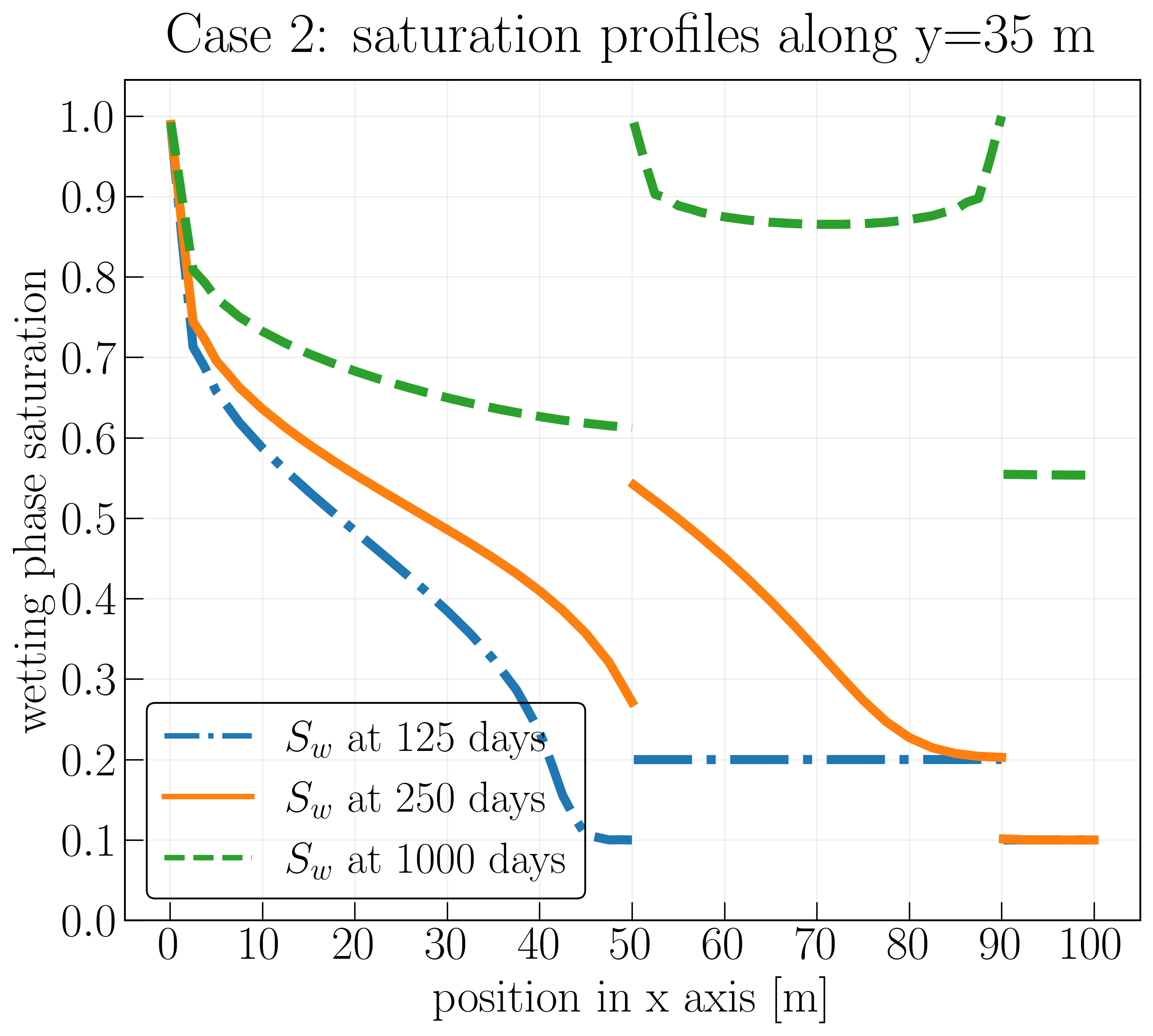}}
\subfigure{\includegraphics[width=0.45\linewidth]{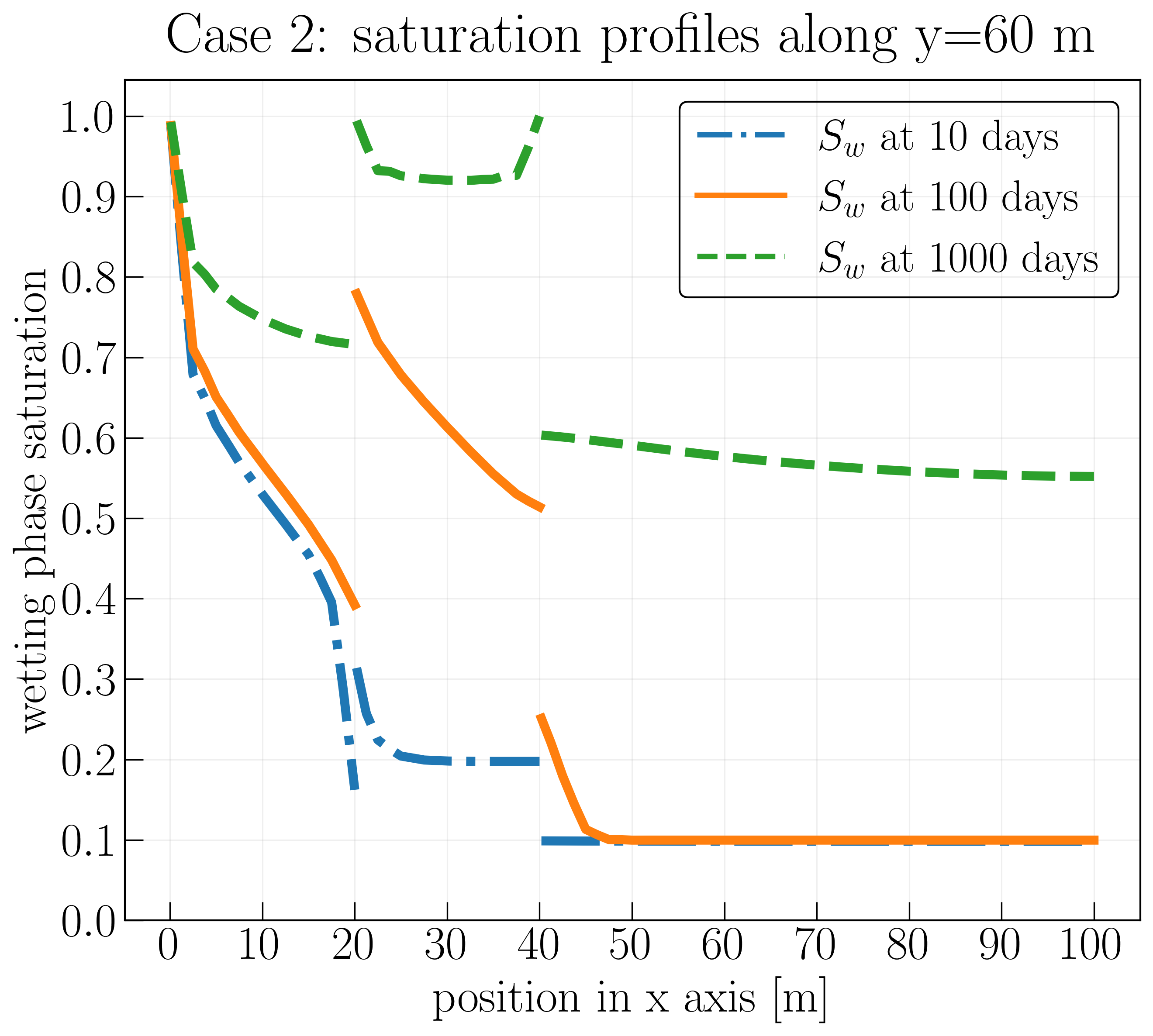}}
\caption{Heterogeneous inclusions problem for Case 2:   wetting phase saturation profiles along $y=35$ m (left) and $y=60$m (right) at selected times.}
\label{fig:block_case2_sat_linechart}
\end{figure}

\begin{figure}[H]
\centering
\includegraphics[width=0.5\linewidth]{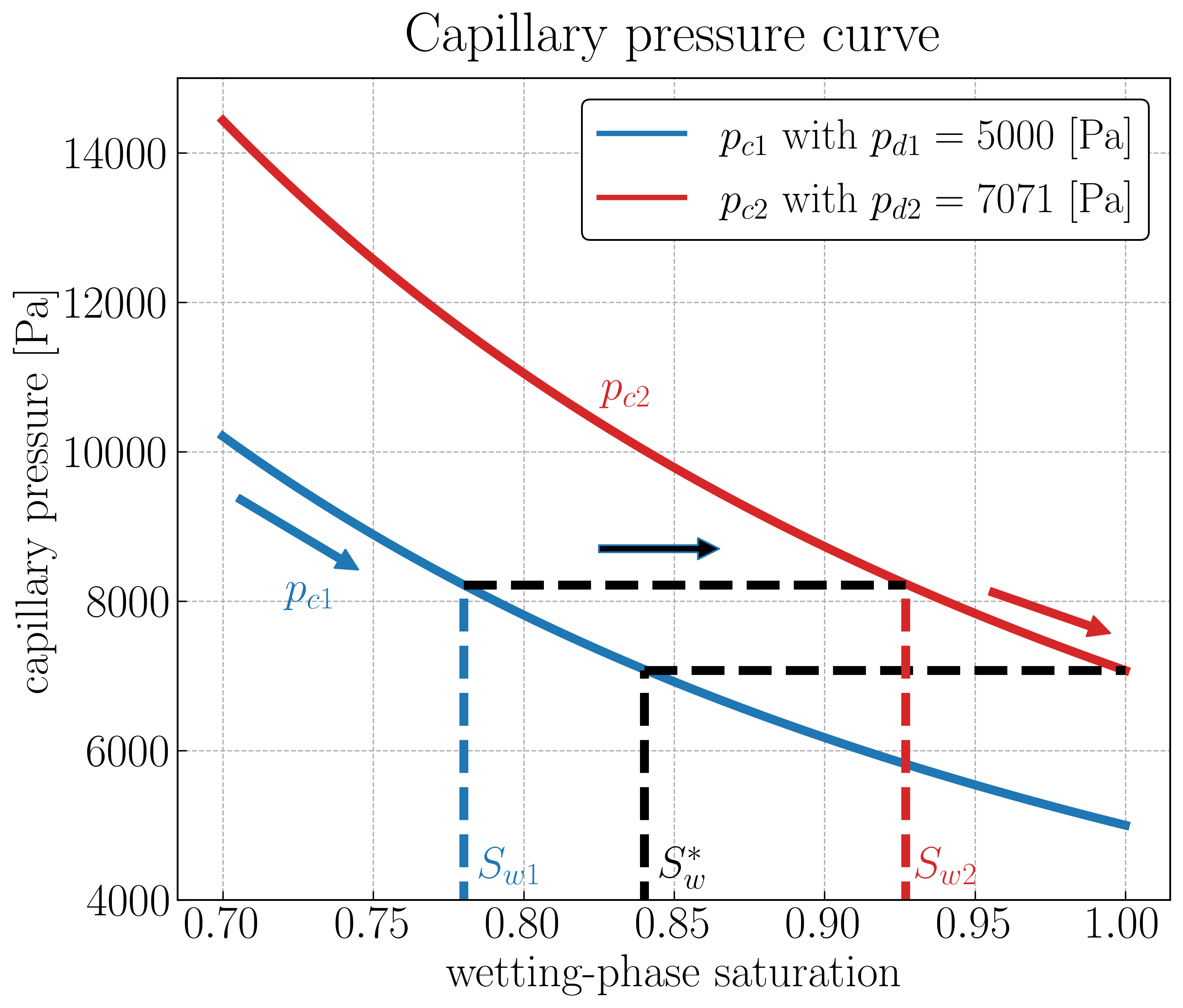}
\caption{Heterogeneous inclusions problem for Case 2: capillary pressure functions for two rocks.} 
\label{fig:blocks_case2_cap_pres} 
\end{figure}

\begin{figure}[H]
\centering     
\subfigure[$t=50$]{\includegraphics[width=0.29\linewidth]{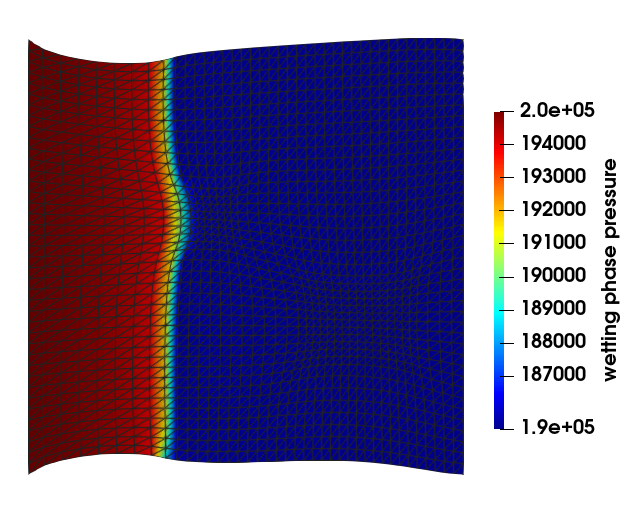}}
\subfigure[$t=125$]{\includegraphics[width=0.29\linewidth]{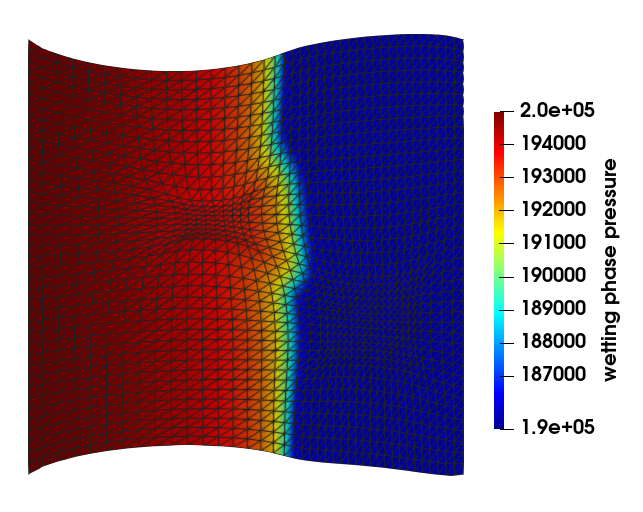}}
\subfigure[$t=250$]{\includegraphics[width=0.29\linewidth]{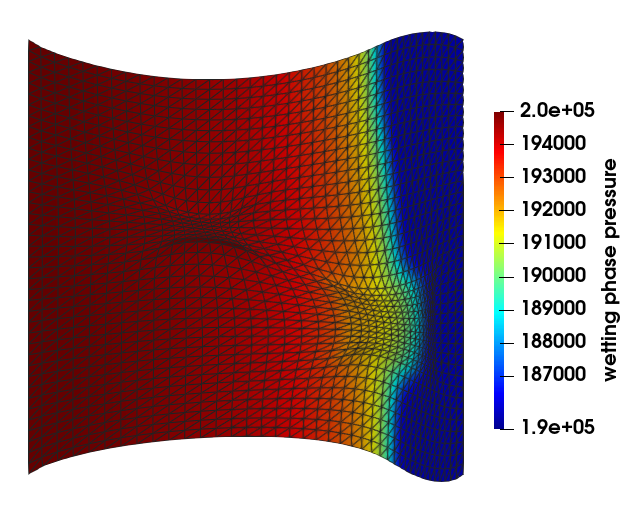}}\\
\subfigure[$t=375$]{\includegraphics[width=0.29\linewidth]{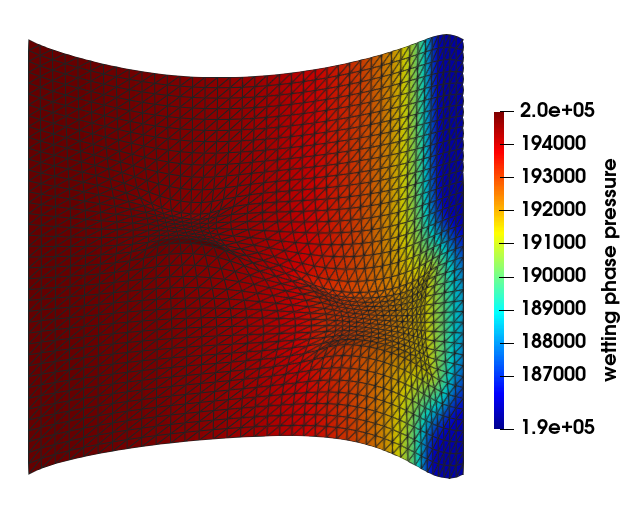}}
\subfigure[$t=500$]{\includegraphics[width=0.29\linewidth]{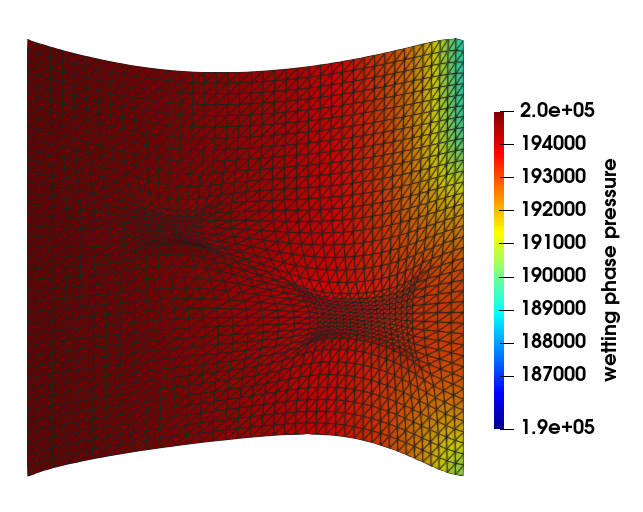}}
\subfigure[$t=1000$]{\includegraphics[width=0.29\linewidth]{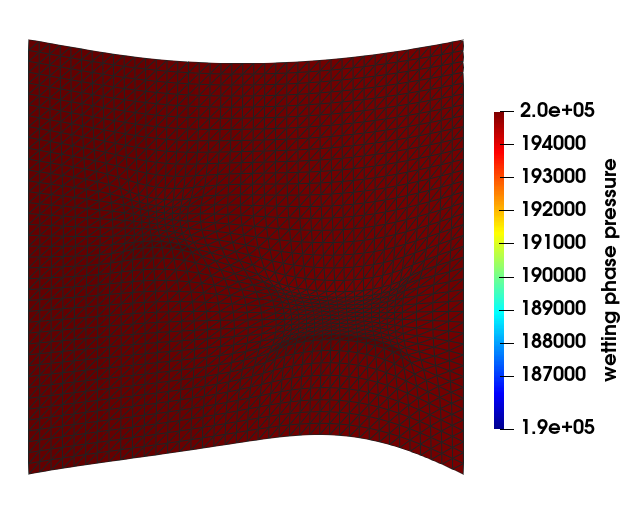}}
\caption{Heterogeneous inclusions problem for Case 2: wetting phase  pressure contours  at $t=50, 125,  250,  375,  500$  and $1000$ days.}
\label{fig:blocks_case2_pres}
\end{figure}

\begin{figure}[H]
\centering     
\subfigure[$t=50$]{\includegraphics[width=0.29\linewidth]{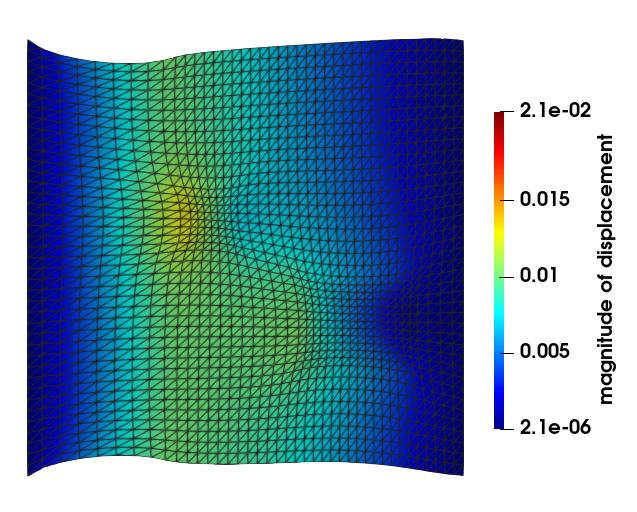}}
\subfigure[$t=125$]{\includegraphics[width=0.29\linewidth]{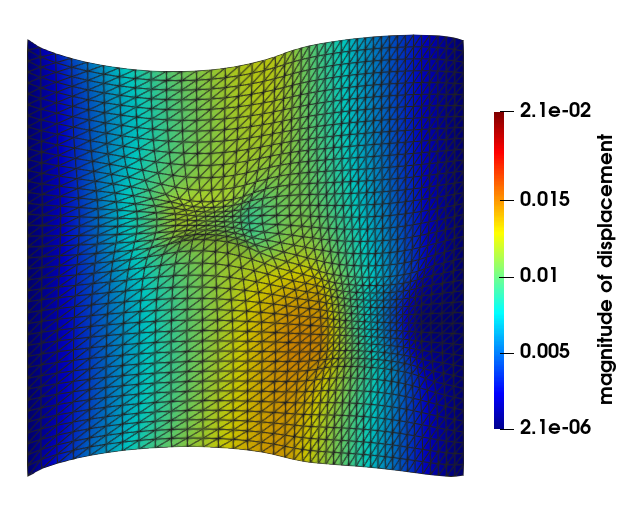}}
\subfigure[$t=250$]{\includegraphics[width=0.29\linewidth]{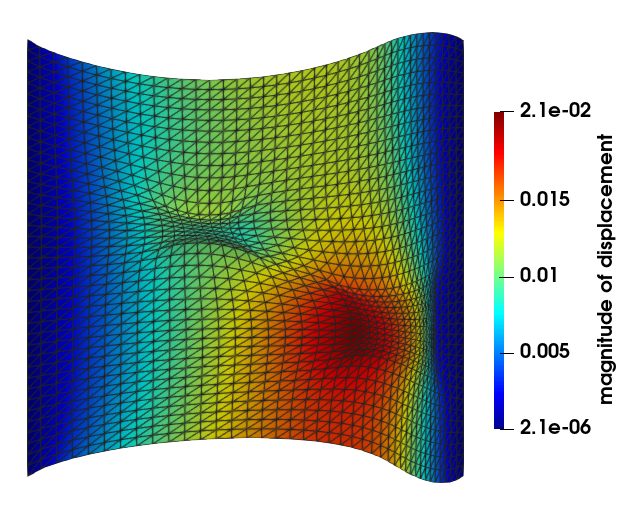}}\\
\subfigure[$t=375$]{\includegraphics[width=0.29\linewidth]{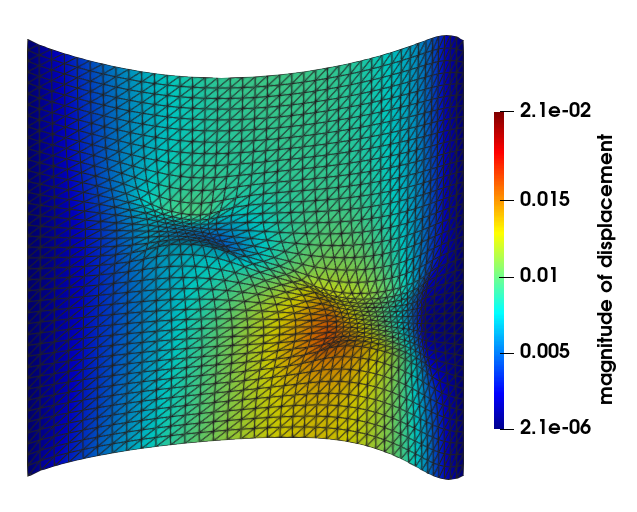}}
\subfigure[$t=500$]{\includegraphics[width=0.29\linewidth]{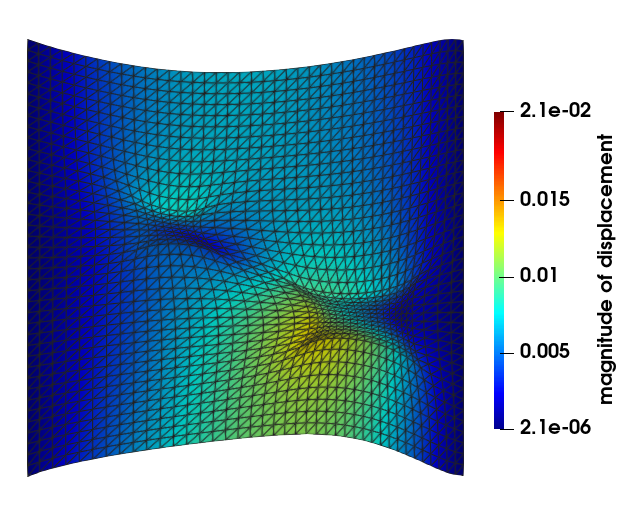}}
\subfigure[$t=1000$]{\includegraphics[width=0.29\linewidth]{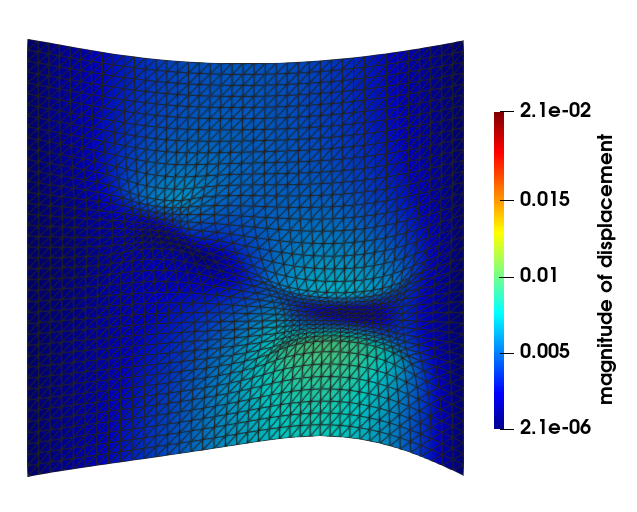}}
\caption{Heterogeneous inclusions problem for Case 2: magnitude of displacement at $t=50, 125,  250,  375,  500$  and $1000$ days.}
\label{fig:blocks_case2_displacement}
\end{figure}

\begin{figure}[H]
\centering     
\subfigure[Case 1]{\includegraphics[width=2.9in]{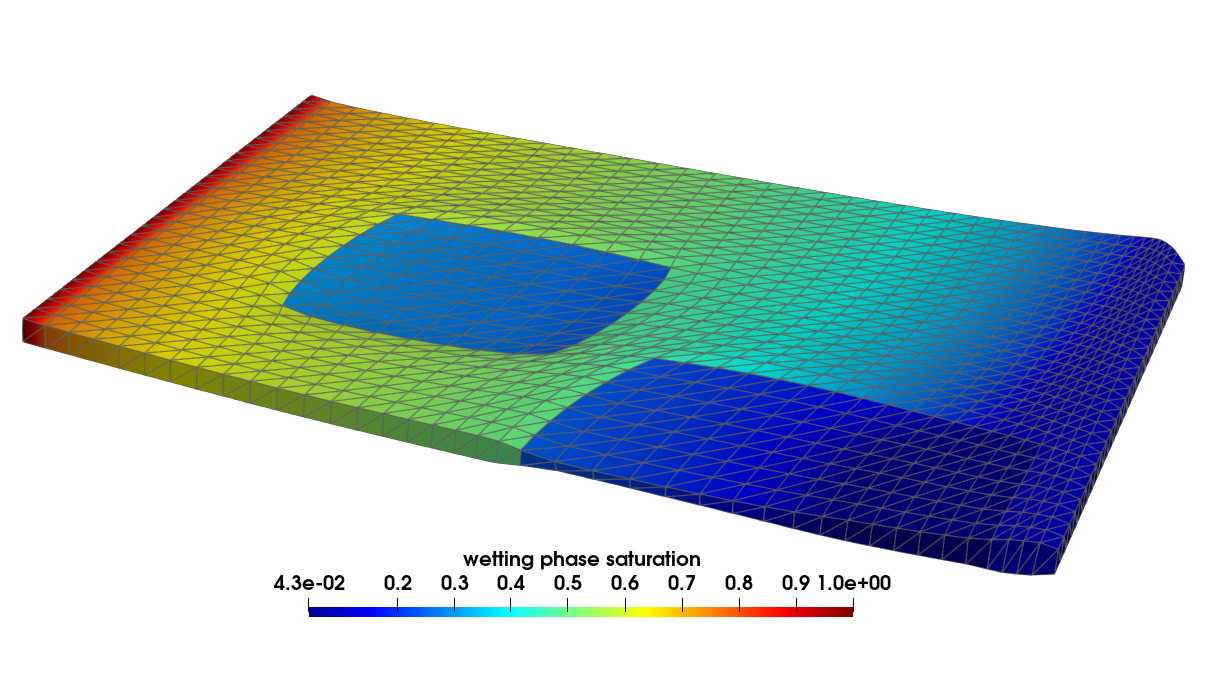}}
\subfigure[Case 2]{\includegraphics[width=2.9in]{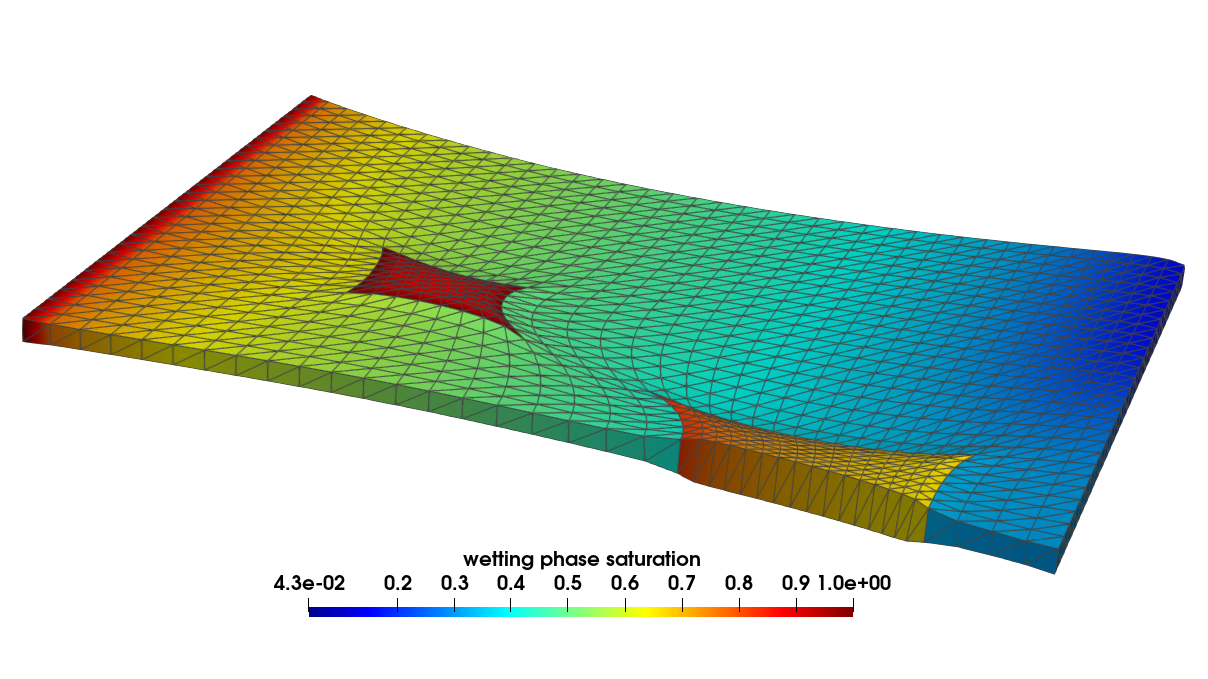}}
\caption{Heterogeneous inclusions problem: 3D views of a cross-section of the domain along the line $y=35$m. Contours correspond to
the wetting phase saturation at $t=500$  days.}
\label{fig:3dzdirection}
\end{figure}

\subsection{Porous Medium Subjected to Load}

The numerical examples in this section show the impact of loading on the wetting phase propagation in the medium as it undergoes deformations. 
The domain $\Omega = [0,100]\times [0,100]\times [0,5]$ m$^3$ is partitioned into  $2400$ tetrahedra.  
Boundary conditions for flow and displacement are described in Fig.~\ref{fig:loading_traction_y_setup}. Dirichlet data
is prescribed for the pressures ($p_{w\D} = 195000$ Pa and $p_{o\D} = 200000$ Pa) on the left side of the boundary and no flow
is imposed on the remainder of the boundary. Two different loading scenarios are considered: first a non-zero traction boundary condition in the $y$-direction is imposed on the top side ($\mathbf{g}_\mathbf{u} =  (0,-r,0)$); this case is referred to as $y-$load. Second a load is imposed in the $x$-direction on the left side of the domain ($\mathbf{g}_\mathbf{u} = (r,0,0)$); this case is referred to as $x-$load.  In both cases, the bottom side is fixed, with zero Dirichlet boundary condition for the displacement. Zero traction is imposed on the remainding of the boundary. The load increases linearly in time:
\[
r(t) = 50000 \frac{t}{T}.
\]

\begin{figure}[H]
\subfigure[flow BCs \label{fig:loadingBCflow}]{
\includegraphics[width=0.3\linewidth]{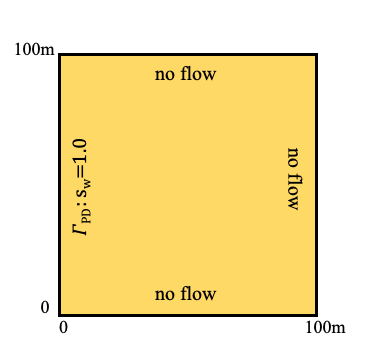}}
\subfigure[$y$-load BCs \label{fig:loadingBCdisp-y}]{
\includegraphics[width=0.3\linewidth]{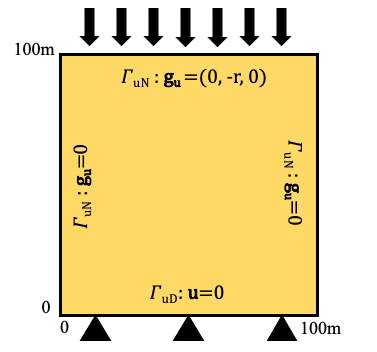}}
\subfigure[$x$-load BCs \label{fig:loadingBCdisp-x}]{
\includegraphics[width=0.3\linewidth]{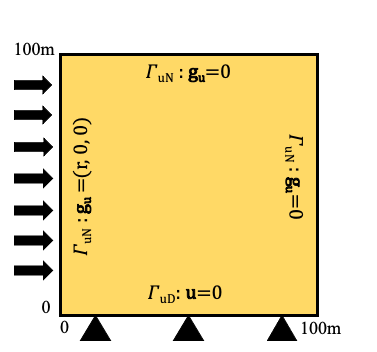}}
\caption{Set-up of boundary conditions for flow  and geomechanics.}
\label{fig:loading_traction_y_setup}
\end{figure}

The following physical parameters are used:
\[
K_w=K_{o}= 10^{4}, \quad K=8.0\times 10^{-11} \mbox{m}^2, \quad \lambda=\mu=4\times 10^5 \, \mbox{Pa},\quad K_s=666666 \, \mbox{Pa}.
\]
We choose smaller  values for the bulk moduli to show the impact of the loading on fluid and solid phases. 
The final time is $T=500$ days and the other computational parameters are 
as in \eqref{eq:comppar}. 

We first show the contours for wetting phase saturation and pressure at $250$, $375$ and $500$ days in Fig.~\ref{fig:yloadsatpres}
for the case of vertical load.
As the load increases, the domain is compressed in the $y-$direction as expected and slightly expanded in the $x-$direction. Even though the pressure gradient is mostly in the $x-$direction, the deformation of the medium creates a small pressure gradient
in the $y-$direction near the load boundary. The wetting phase floods the top part of the domain slower than the bottom part. 
\begin{figure}[H]
\centering
\subfigure[$S_w, \, t=250$]{\includegraphics[height=0.25\linewidth,keepaspectratio]{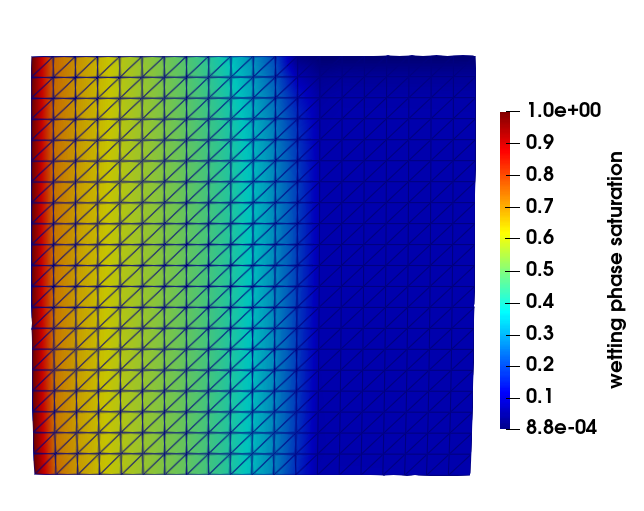}}
\subfigure[$S_w, \, t=375$]{\includegraphics[height=0.25\linewidth,keepaspectratio]{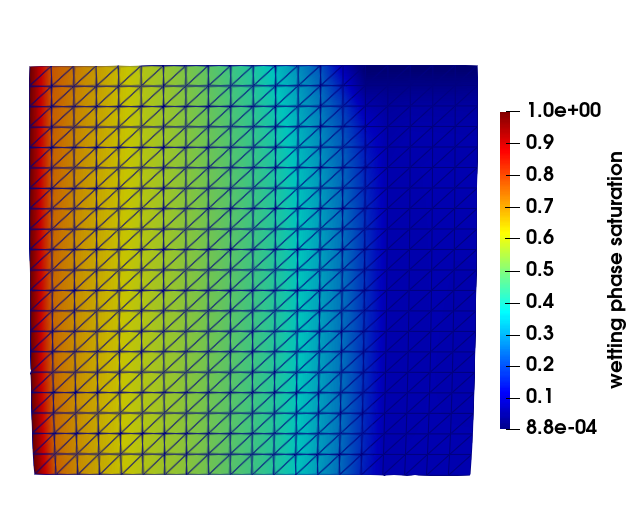}}
\subfigure[$S_w, \, t=500$]{\includegraphics[height=0.25\linewidth,keepaspectratio]{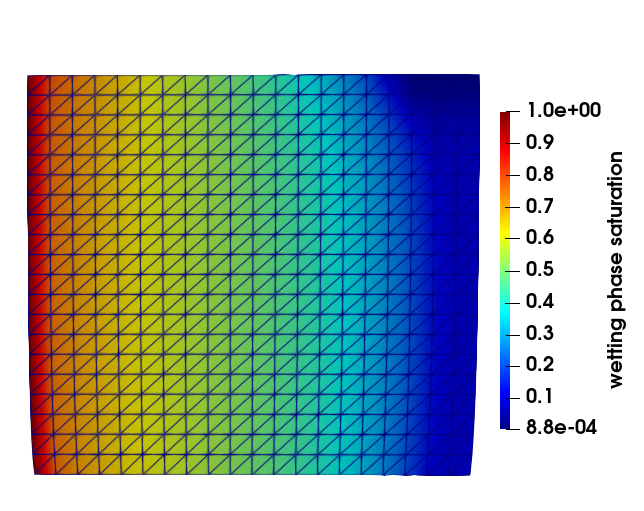}}\\
\subfigure[$P_w, \, t=250$]{\includegraphics[height=0.25\linewidth,keepaspectratio]{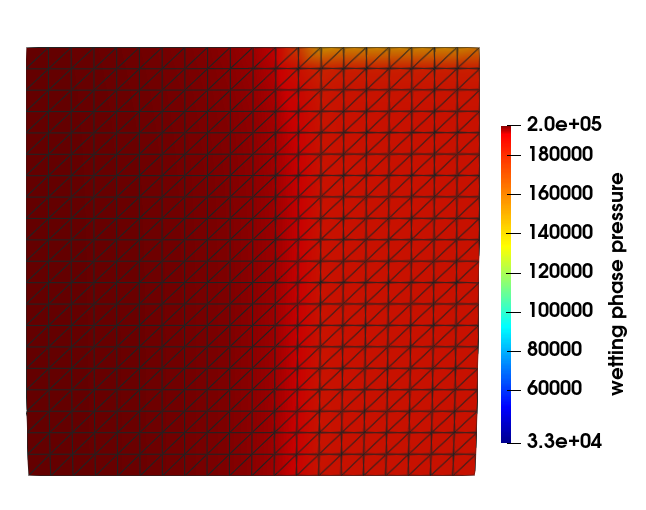}}
\subfigure[$P_w, \, t=375$]{\includegraphics[height=0.25\linewidth,keepaspectratio]{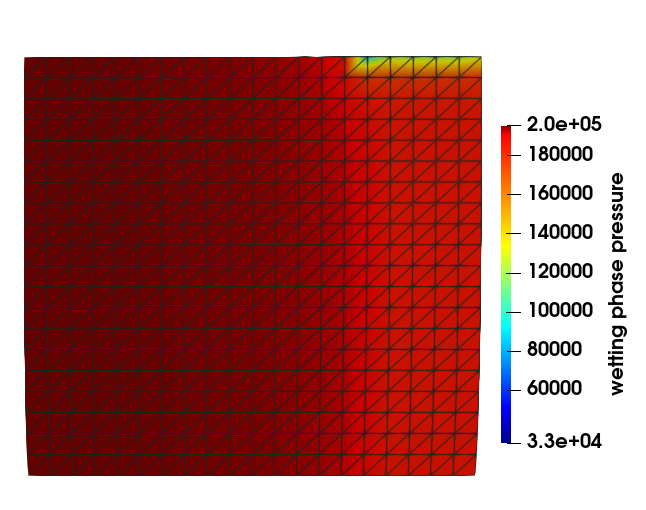}} 
\subfigure[$P_w, \, t=500$]{\includegraphics[height=0.25\linewidth,keepaspectratio]{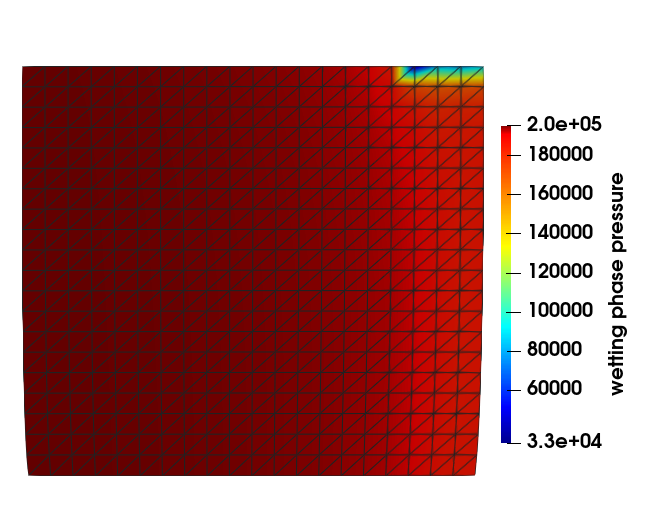}} 
\caption{Case of $x$-load: wetting phase saturation and pressure contours at different times}
\label{fig:yloadsatpres}
\end{figure}
To better see this, we extract the saturation profiles at $250, 375$ and $500$ days along three horizontal lines (see Fig.~\ref{fig:yload-profiles3D}). The location of the front is also indicated in the figure.  Near the top side of the domain, the saturation front is lagging behind by ten meters.
\begin{figure}[H]
\centering
\subfigure{\includegraphics[width=0.75\linewidth]{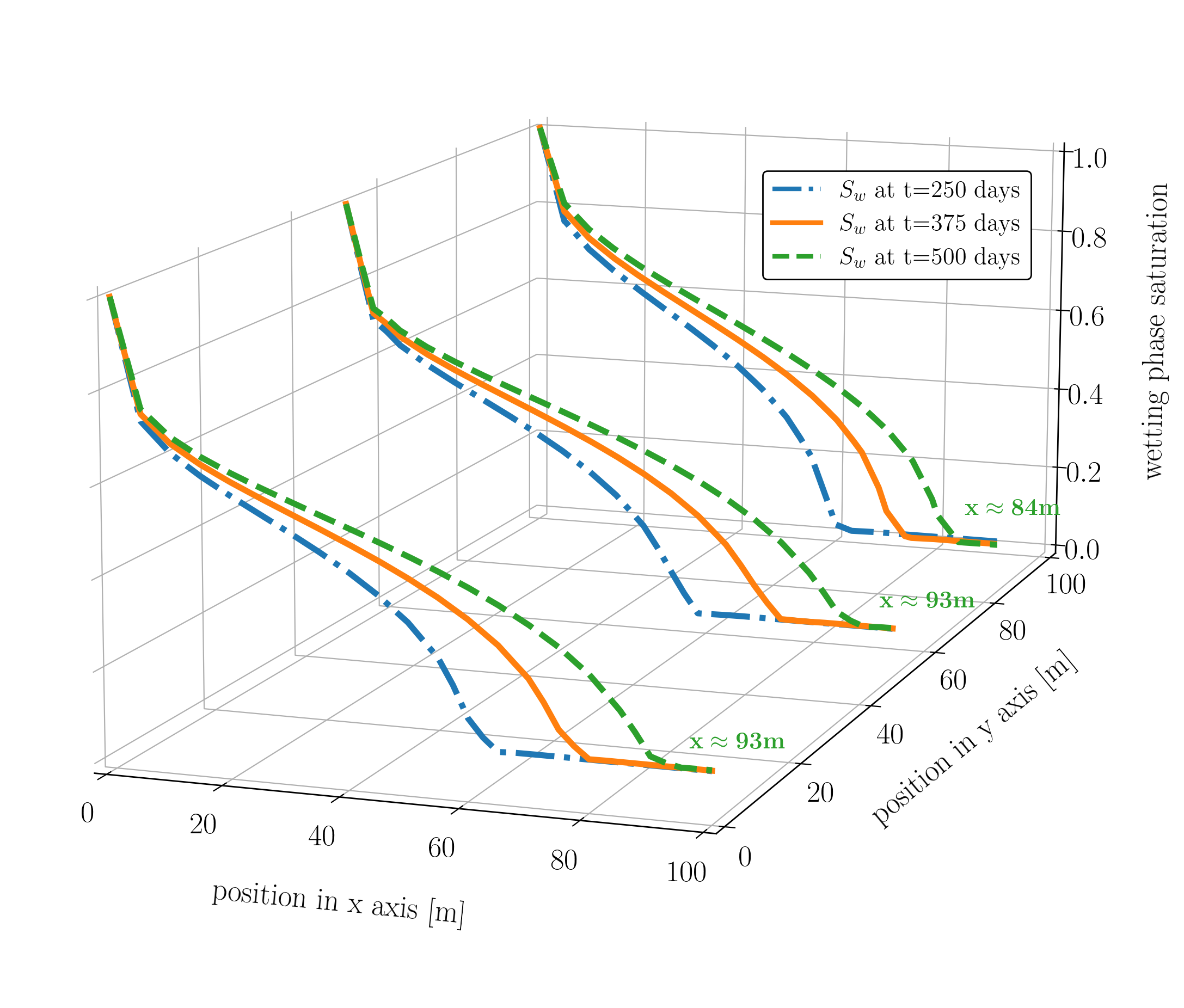}}
\caption{Case of $y$-load: wetting phase saturation profiles along $y=0$, $50$ and $100$ m at $t=250$, $375$ and $500$ days.}
\label{fig:yload-profiles3D}
\end{figure}
Next, we show the saturation and pressure contours for the case of $x-$load in Fig.~\ref{fig:xloadsatpres}.
\begin{figure}[H]
\centering
\subfigure[$S_w, \, t=250$]{\includegraphics[height=0.25\linewidth,keepaspectratio]{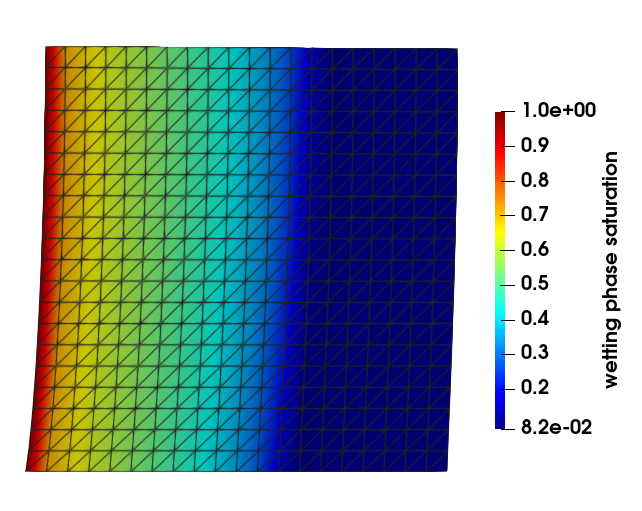}}
\subfigure[$S_w, \, t=375$]{\includegraphics[height=0.25\linewidth,keepaspectratio]{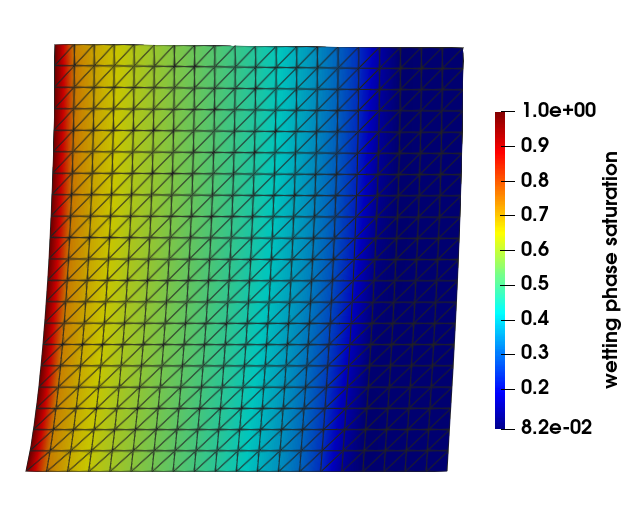}}
\subfigure[$S_w, \, t=500$]{\includegraphics[height=0.25\linewidth,keepaspectratio]{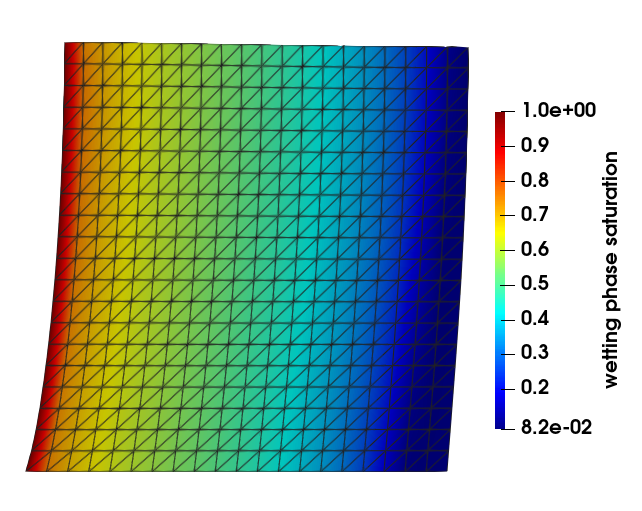}}\\
\subfigure[$P_w, \, t=250$]{\includegraphics[height=0.25\linewidth,keepaspectratio]{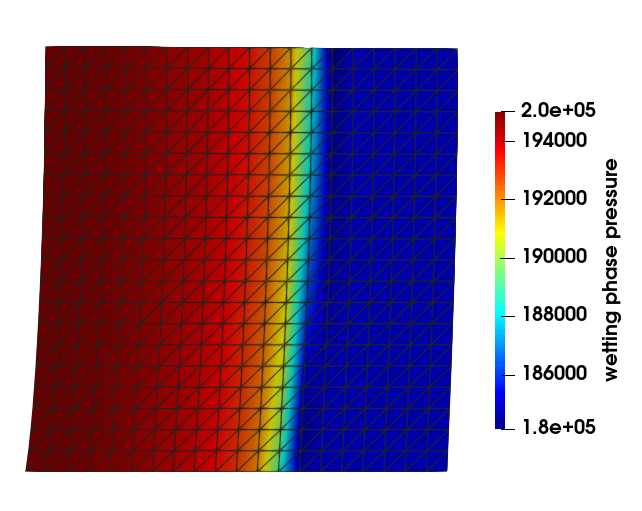}}
\subfigure[$P_w, \, t=375$]{\includegraphics[height=0.25\linewidth,keepaspectratio]{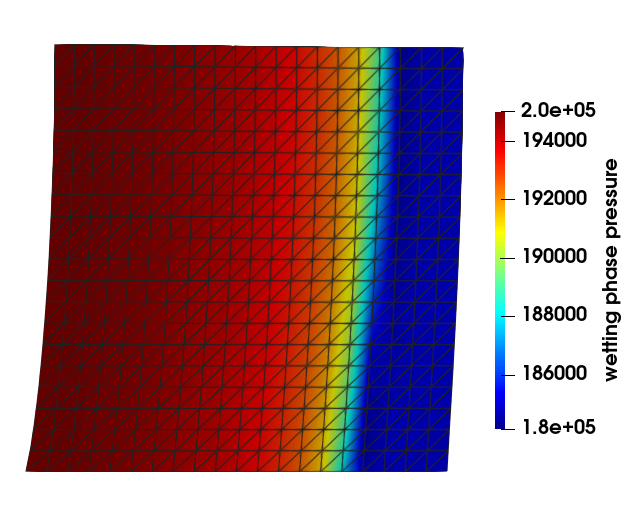}} 
\subfigure[$P_w, \, t=500$]{\includegraphics[height=0.25\linewidth,keepaspectratio]{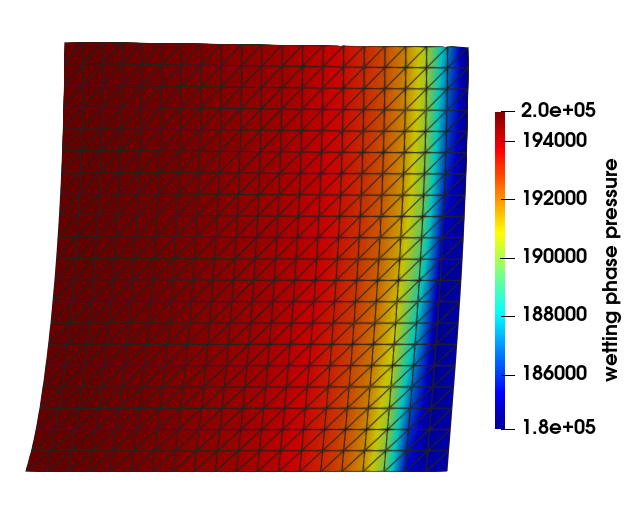}} 
\caption{Case of $x$-load: wetting phase saturation and pressure contours at different times.}
\label{fig:xloadsatpres}
\end{figure}
In this loading scenario, the deformation of the medium is mostly in the $x-$direction, with the top part of the domain deforming the most because of the constraint of zero displacement at the bottom side. We also observe that the displacement of the domain is in the same direction than the propagation of the wetting phase saturation.  This yields a faster saturation front in the top part of the domain. Fig.~\ref{fig:xload-profiles3D} shows the saturation profiles along three horizontal lines.  After $500$ days, the saturation front at the top side reaches about 97 meters which is 3 and 8 meter further than other two locations.

\begin{figure}[H]
\centering
\subfigure{\includegraphics[width=0.75\linewidth]{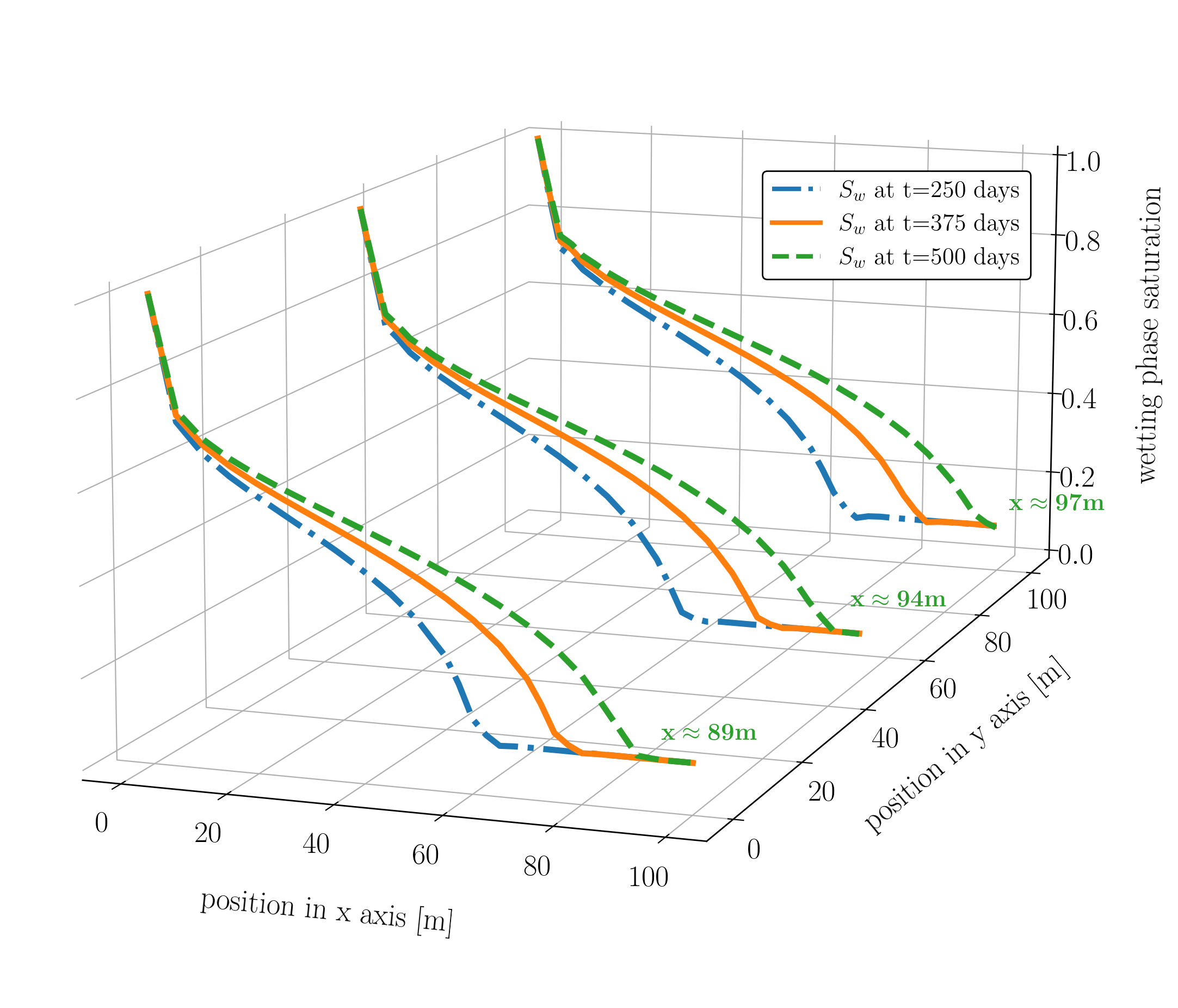}}
\caption{Case of $x$-load: wetting phase saturation profiles along $y=0$, $50$ and $100$ m at $t=250$, $375$ and $500$ days.}
\label{fig:xload-profiles3D}
\end{figure}

For a better comparison between these two types of loading, we show the contours of the $x-$ and $y-$ components of the displacement 
at the final time in Fig.~\ref{fig:xydispT}. Under the $y-$load, the medium is compressed vertically and stretched horizontally whereas
under the $x-$load, the medium deforms mostly along the direction of the flow except for the fixed bottom boundary.
\begin{figure}
\centering
\subfigure[$y$-load: $U_x$]{\includegraphics[height=0.22\textheight,keepaspectratio]{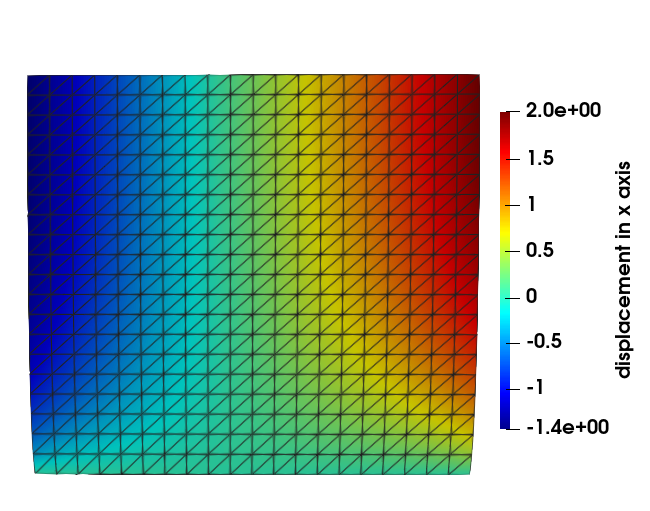}} 
\subfigure[$y$-load: $U_y$]{\includegraphics[height=0.22\textheight,keepaspectratio]{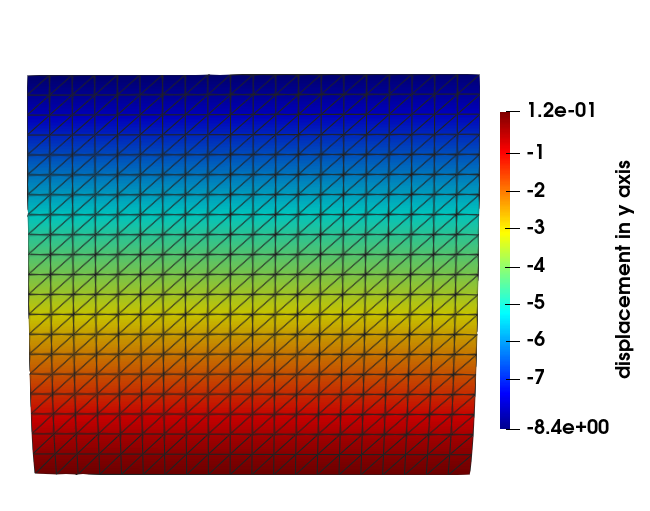}}\\
\subfigure[$x$-load: $U_x$]{\includegraphics[height=0.22\textheight,keepaspectratio]{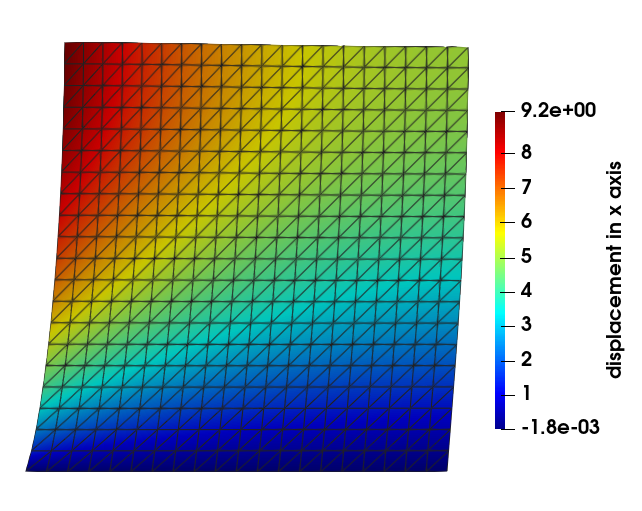}} 
\subfigure[$x$-load: $U_y$]{\includegraphics[height=0.22\textheight,keepaspectratio]{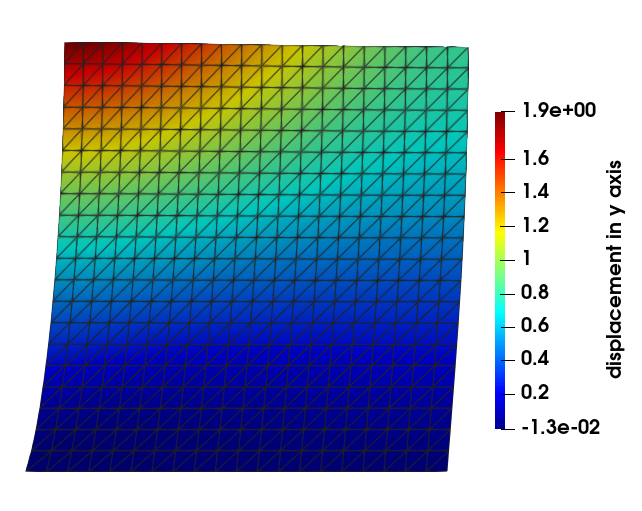}}
\caption{Contours of $x$ and $y$ components of displacement at $500$ days.}
\label{fig:xydispT}
\end{figure}

Finally, we now compare the effect of no loading  versus loading for both $y-$ and $x-$ loads.
To be precise, no loading means that zero
traction boundary condition ($\mathbf{g}_\mathbf{u} = {\bf 0}$) is prescribed  on the boundary except for the bottom boundary where zero displacement is imposed. 
Fig.~\ref{fig:loadvnoload} shows the wetting phase saturation profiles extracted along the top and bottom sides  at $250$, $375$ and $500$ days.  
On the top boundary, we observe that the saturation front advances faster in the $x-$load than in the zero traction case and the $y-$load
yields the slowest saturation front. This is expected since the loading direction for the $x-$load is the same as the flow direction. 
On the bottom boundary, overall there are less differences between the profiles for the three loading scenarios because of the zero displacement constraint. This figure shows the impact of the nonlinearities in the problem on the fluid propagation.

\begin{figure}
\subfigure[Along $y=100$ m]{\includegraphics[width=0.5 \linewidth]{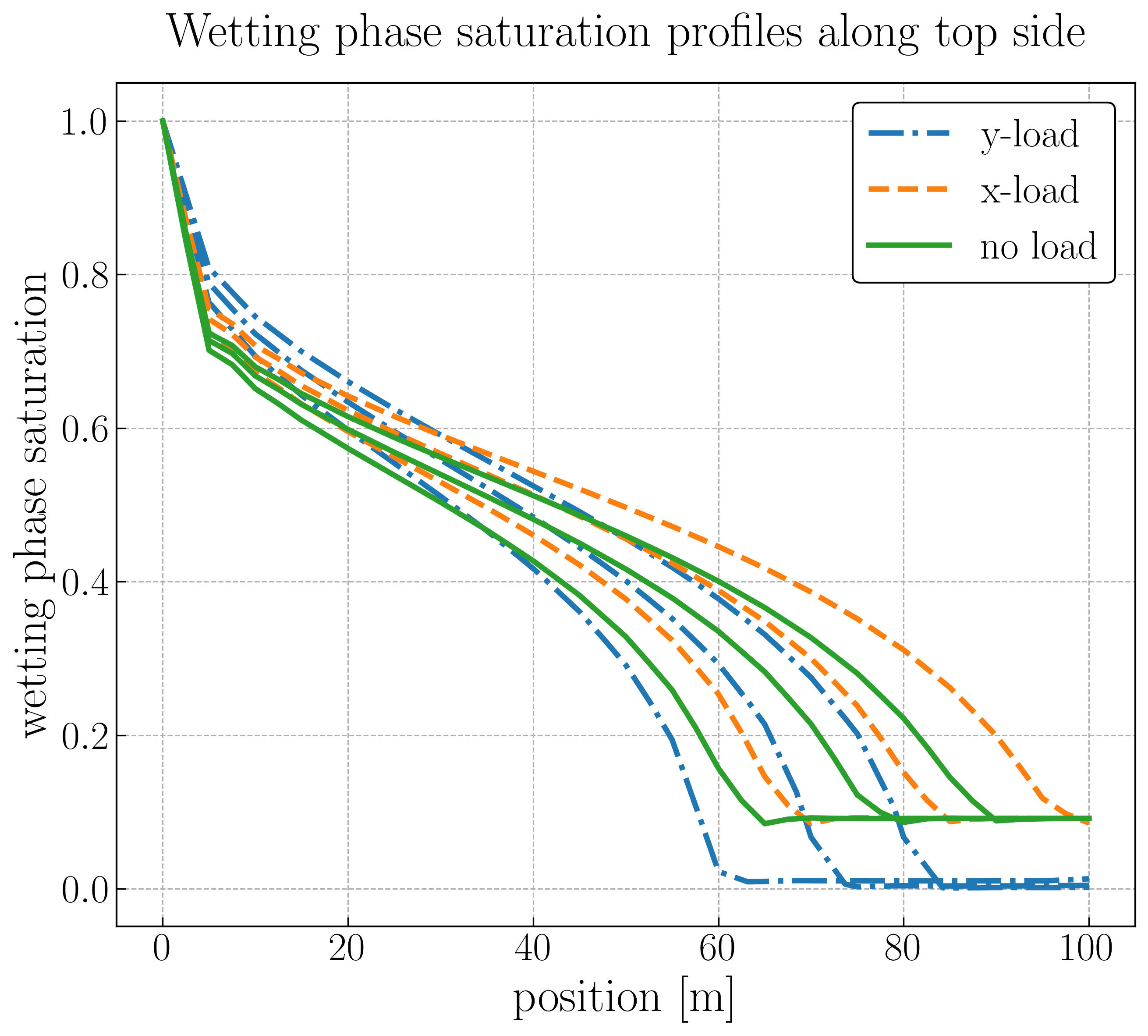}}
\subfigure[Along $y=0$ m]{\includegraphics[width=0.5 \linewidth]{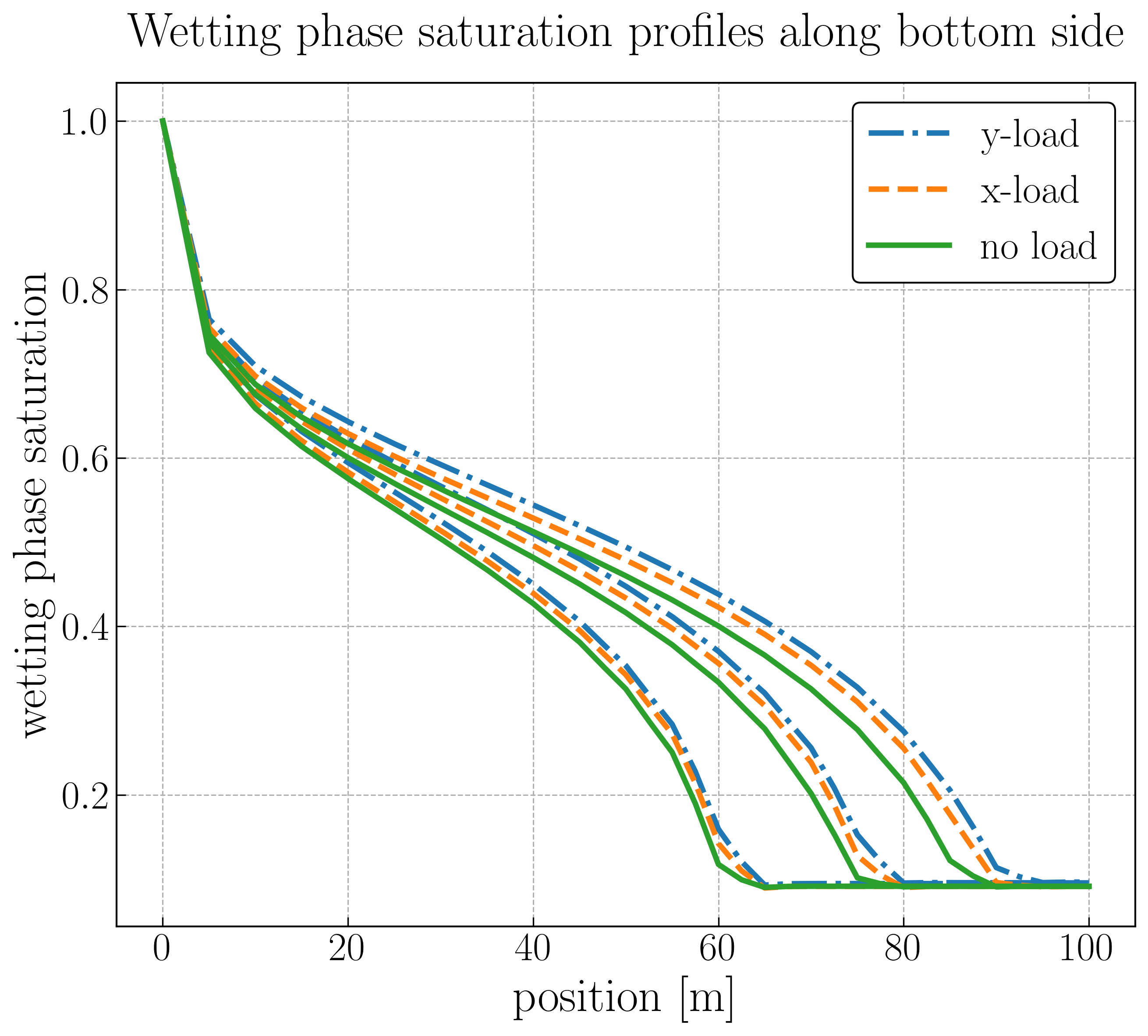}}
\caption{Wetting phase saturation profiles extracted along $y=100$m and $y=0$m at three different times: $250$, $375$ and $500$ days and
for different loading scenarios.}
\label{fig:loadvnoload}
\end{figure}

\subsection{Highly Heterogeneous Medium}

We apply the method to a porous medium where both porosity and permeability vary in space.  
The medium exhibits regions of high permeability (channels) surrounded by regions of low permeability and lower porosity. 
This example demonstrates the capability of the proposed method to handle large variations in permeability.
The domain $[0,80]\times[0,80] \times [0,7.5]$ consists of three stacked horizontal layers of height $2.5$ m.  The mesh contains $18432$ tetrahedra.  
The porosity field for the three layers is shown in Fig.~\ref{fig:speporosity} and the permeability field in logarithmic scale is shown in Fig.~\ref{fig:speperm}.  The data are extracted from the SPE10 porosity and permeability fields; they correspond to a section of layer 43, 44 and 45 in the SPE10 model \cite{SPE10reference}. 
Dirichlet data is prescribed for the pressures ($p_{w\D} = 1950000$ Pa and $p_{o\D} = 2000000$ Pa) on the left side of the boundary and no flow is imposed on the remainder of the boundary.  The entry pressure is $p_d=50000$ Pa.  
The computational parameters are:
\begin{equation}
\tau=20 \mbox{ days}, \quad \tau_0 = 0.2\mbox{ days}, \quad \sigma_p = 800, \quad \sigma_\bfu = 800, \quad \gamma = 10^5,\quad T=4000 \mbox{ days}.
\end{equation}

\begin{figure}[H]
\centering
\includegraphics[width=0.32\linewidth]{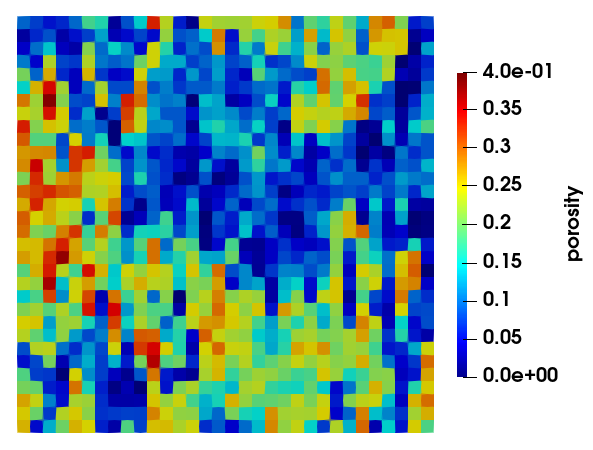}
\includegraphics[width=0.32\linewidth]{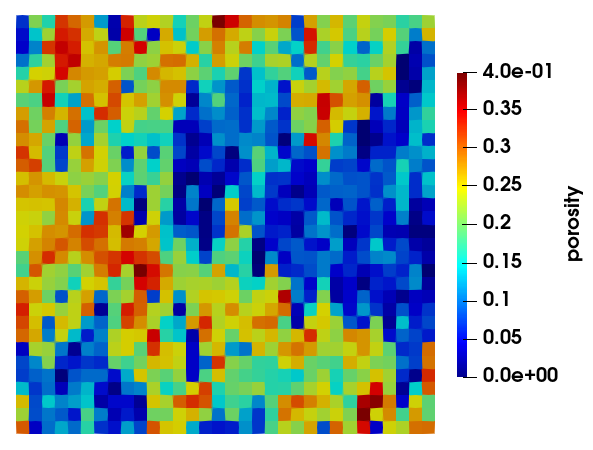}
\includegraphics[width=0.32\linewidth]{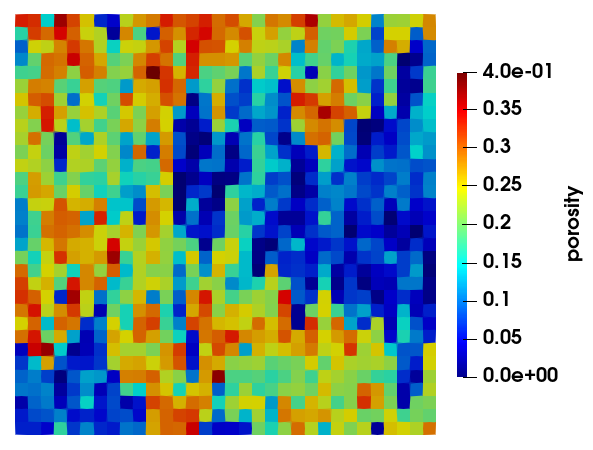}
\caption{Heterogeneous medium: porosity field for bottom layer (left), middle layer (center) and top layer (right).}
\label{fig:speporosity}
\end{figure}
\begin{figure}[H]
\centering
\includegraphics[width=0.32\linewidth]{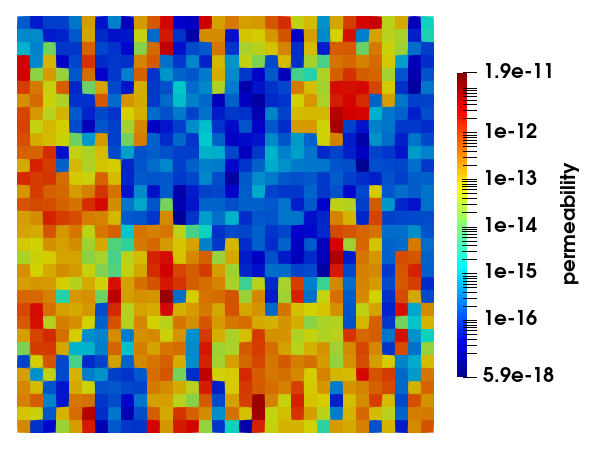}
\includegraphics[width=0.32\linewidth]{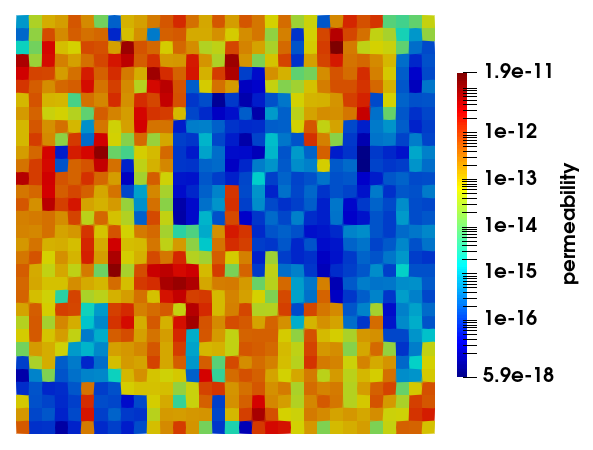}
\includegraphics[width=0.32\linewidth]{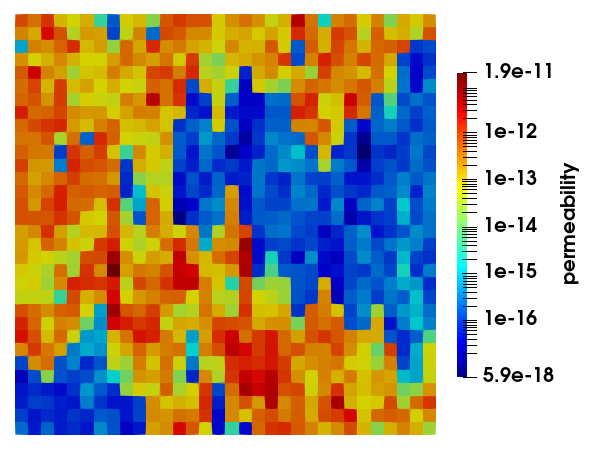}
\caption{Heterogeneous medium: permeability field in log scale for bottom layer (left), middle layer (center) and top layer (right).}
\label{fig:speperm}
\end{figure}
 
Fig.~\ref{fig:3Dsat} shows the wetting phase saturation in the three-dimensional domain at time $t= 1000$ days; values of the saturation above 0.21 are shown only.  We observe a non-uniform saturation front.  The deformation of the domain is magnified by a scaling factor of $100$ for visualization. 

\begin{figure}[H]
\centering
\includegraphics[width=0.6\linewidth]{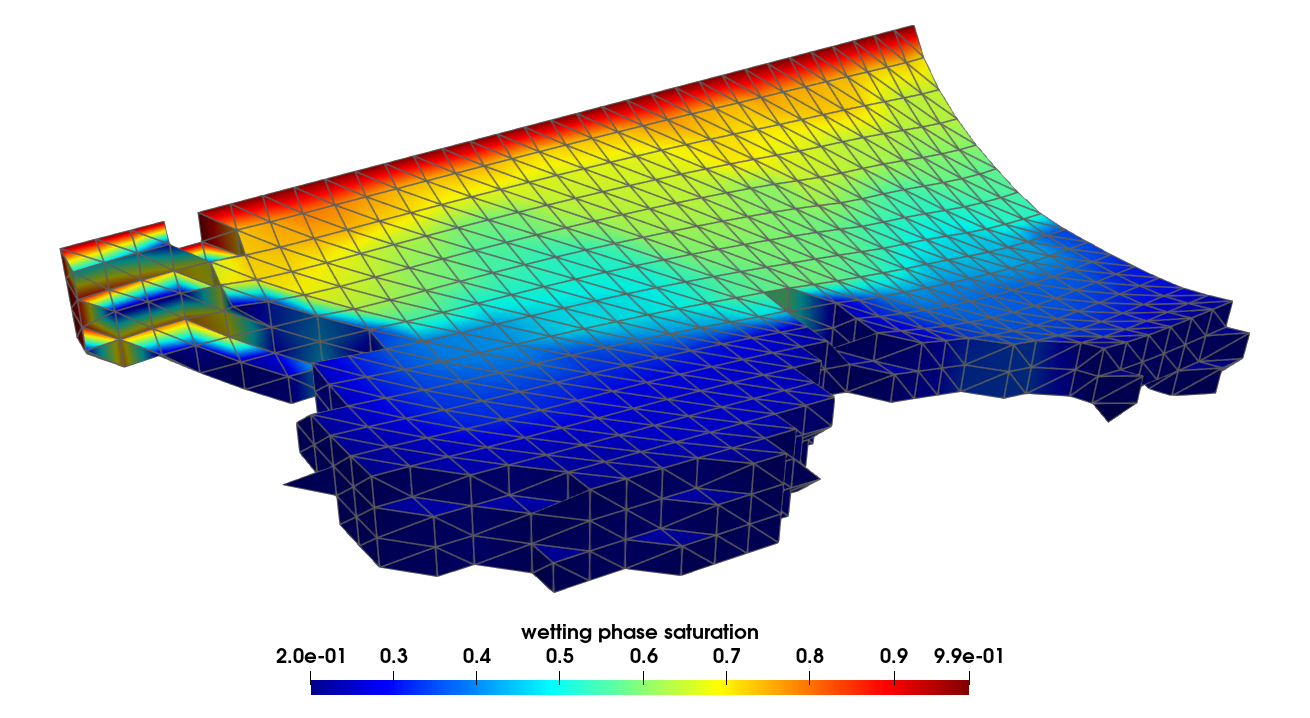}
\caption{Two-phase flow in highly heterogeneous medium. Threshold plot of wetting phase saturation where the value is greater than 0.21 at $t=1000$ days,  displacement scaled up by 100 for visualization.}
\label{fig:3Dsat}
\end{figure}

The wetting phase saturation and pressure at $4000$ days are shown in each of the three layers in Fig.~\ref{fig:spelayersat}. 
For visualization purposes, each component of the numerical approximation of the displacement has been scaled by $100$.
Due to the heterogeneous permeability and porosity, we observe differences in the pressure and saturation contours at each layer.
This simulation shows the effect of three-dimensional heterogeneities in the propagation of the wetting phase through the medium.

\begin{figure}[H]
\centering
\subfigure[$S_w$, top layer]{\includegraphics[width=0.32\linewidth]{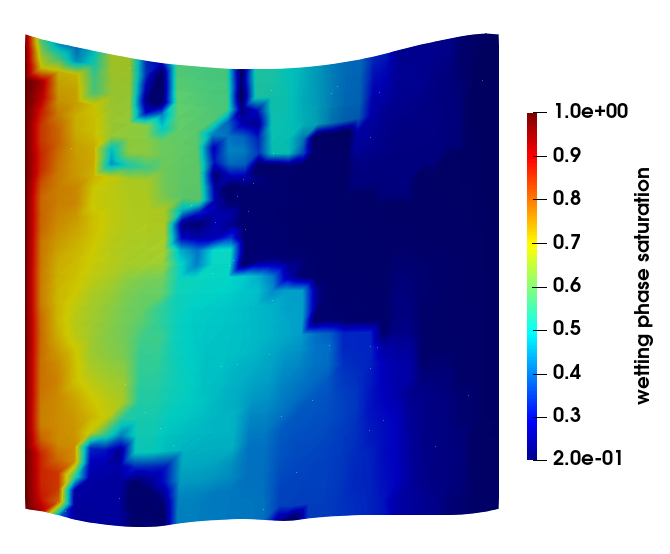}}
\subfigure[$P_w$, top layer]{\includegraphics[width=0.32\linewidth]{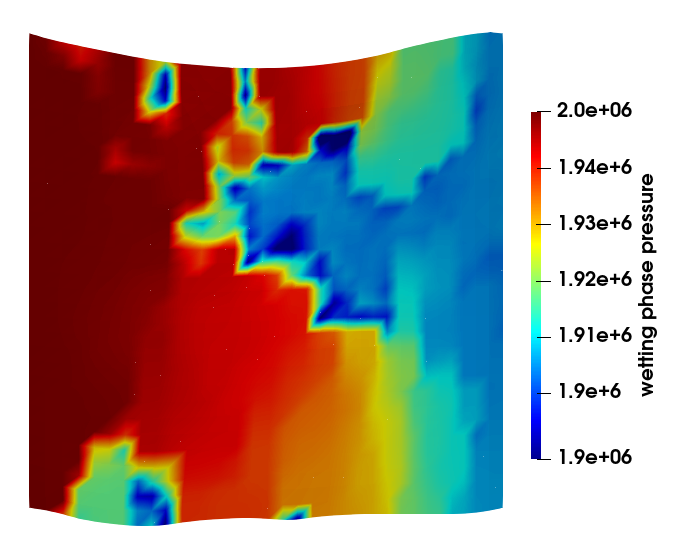}}\\
\subfigure[$S_w$, middle layer]{\includegraphics[width=0.32\linewidth]{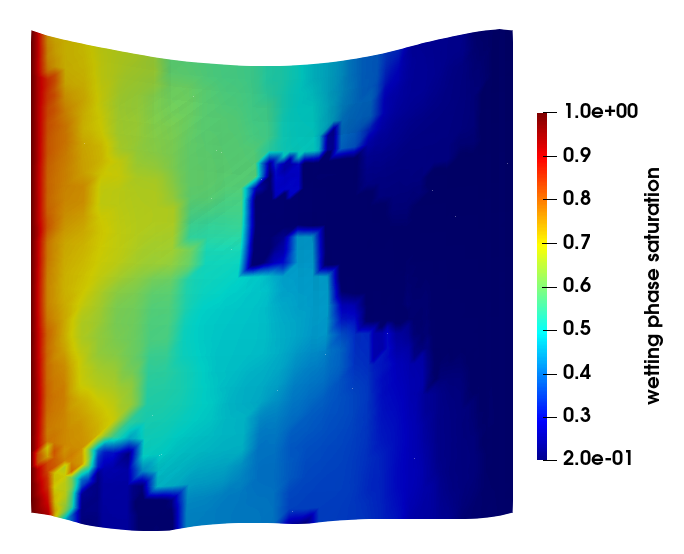}}
\subfigure[$P_w$, middle layer]{\includegraphics[width=0.32\linewidth]{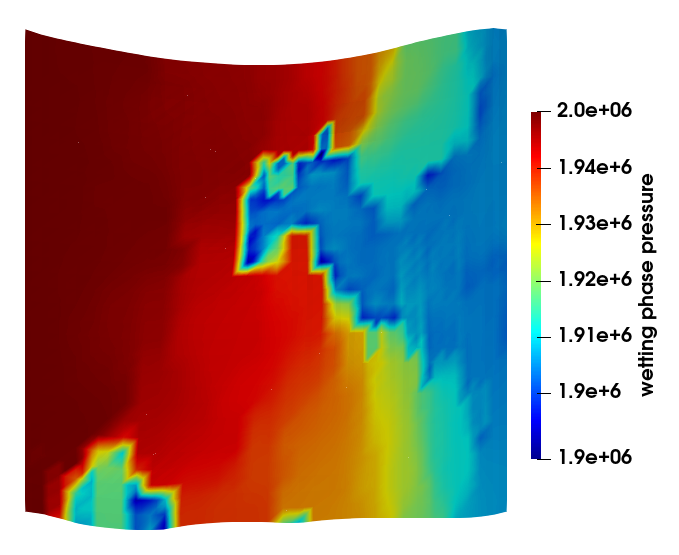}}\\
\subfigure[$S_w$, bottom layer]{\includegraphics[width=0.32\linewidth]{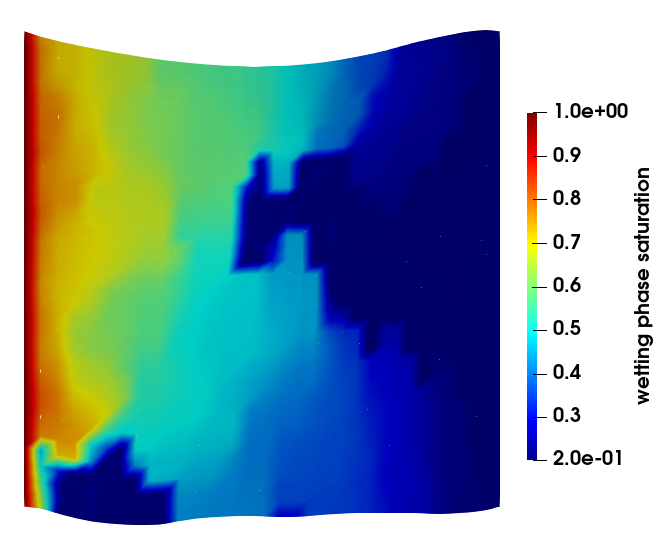}}
\subfigure[$P_w$, bottom layer]{\includegraphics[width=0.32\linewidth]{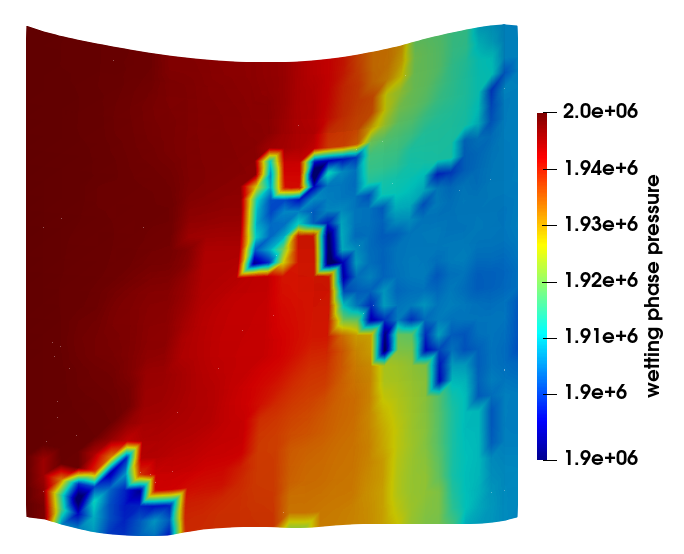}}
\caption{Two-phase flow in highly heterogeneous medium.  Left column: wetting phase saturation in the three layers. Right column: wetting phase pressure,  at t=4000 days.}
\label{fig:spelayersat}
\end{figure}

The contours for the x-, y-, and z-components of the displacement are shown in Fig.~\ref{fig:spe10disp}. The displacement is five
times larger in the flow direction, which is consistent with the choice of the boundary conditions. Because of the coupling between flow
and geomechanics, the displacement components vary in time as the medium is flooded by the wetting phase.
\begin{figure}[H]
\centering
\subfigure[$U_x, t=500$]{\includegraphics[width=0.32\linewidth]{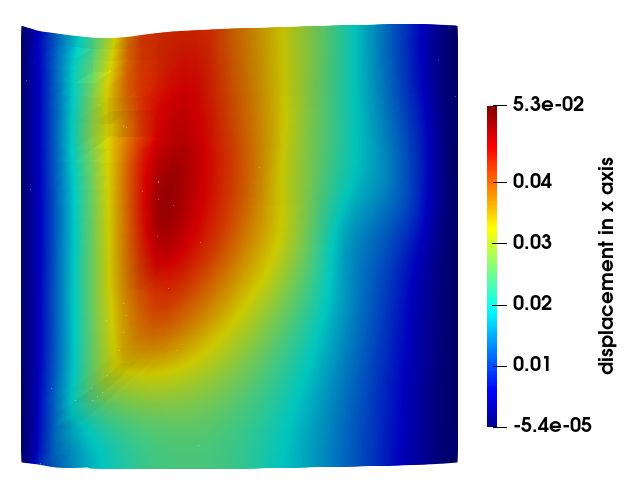}}
\subfigure[$U_y, t=500$]{\includegraphics[width=0.32\linewidth]{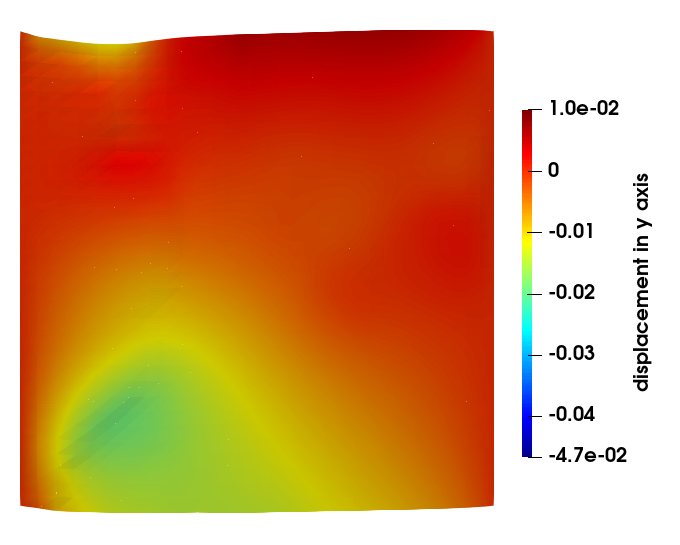}}
\subfigure[$U_z, t=500$]{\includegraphics[width=0.32\linewidth]{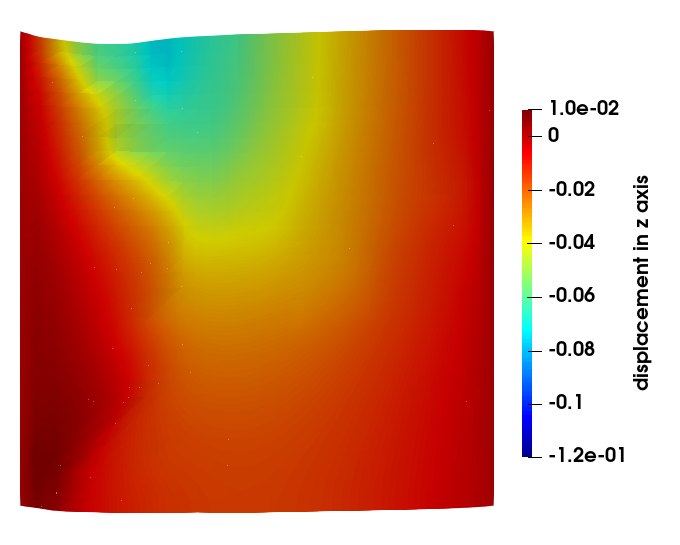}}\\
\subfigure[$U_x, t=2000$]{\includegraphics[width=0.32\linewidth]{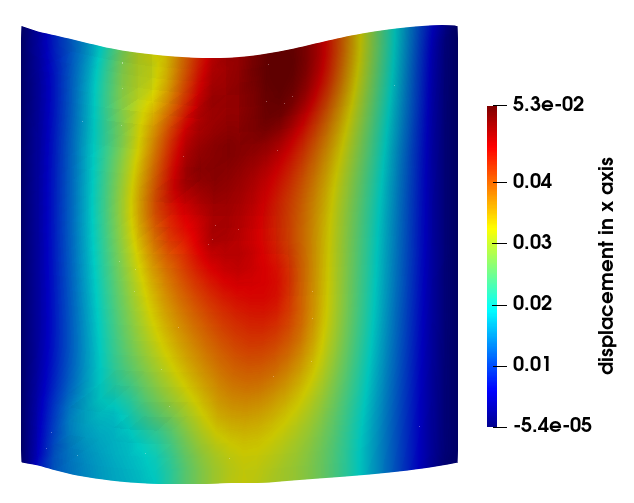}}
\subfigure[$U_y, t=2000$]{\includegraphics[width=0.32\linewidth]{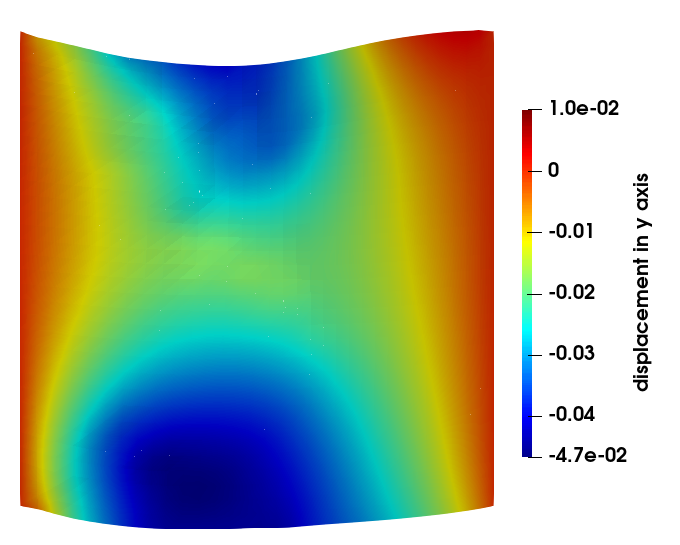}}
\subfigure[$U_z, t=2000$]{\includegraphics[width=0.32\linewidth]{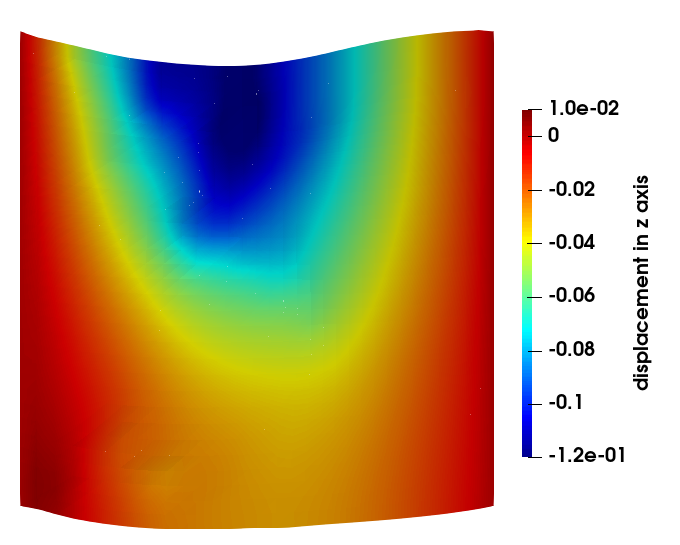}}\\
\subfigure[$U_x, t=4000$]{\includegraphics[width=0.32\linewidth]{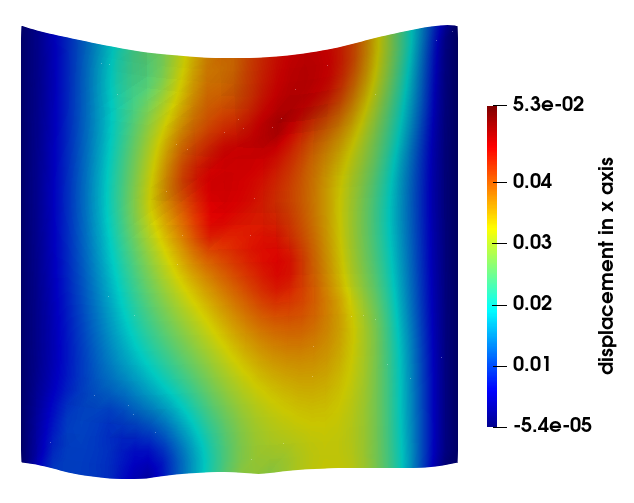}}
\subfigure[$U_y, t=4000$]{\includegraphics[width=0.32\linewidth]{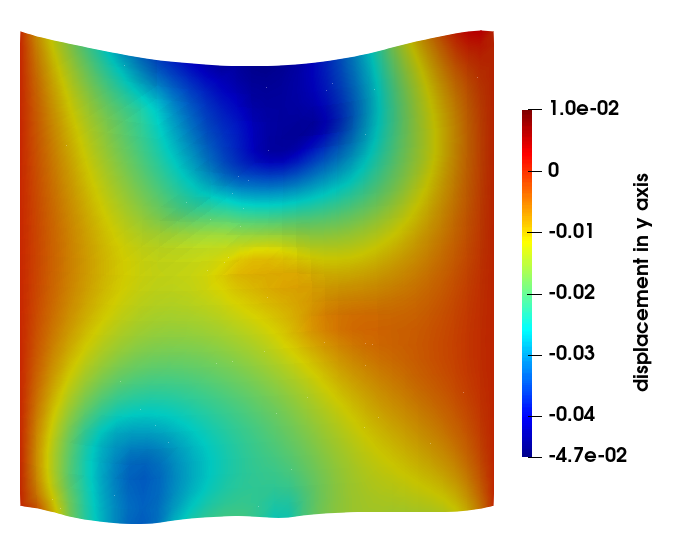}}
\subfigure[$U_z, t=4000$]{\includegraphics[width=0.32\linewidth]{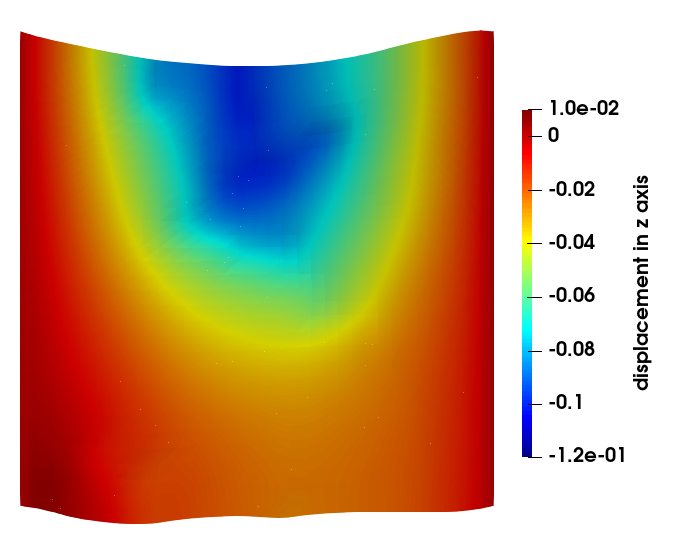}}
\caption{Two-phase flow in highly heterogeneous medium.  Contours of displacement components in top layer at different times:
x-component (left column), y-component (center column) and z-component (right column).}
\label{fig:spe10disp}
\end{figure}

\section{Conclusions}
We have presented an accurate and robust numerical method for solving the coupled two-phase flow and geomechanics equations in porous media.
The method is sequentially implicit, therefore computationally less expensive than a fully implicit scheme. The sequential scheme is stable due to
stabilization terms added to the displacement equation.  The method is validated on three-dimensional benchmark problems and the numerical results confirm
the stability, robustness and accuracy of the proposed scheme for various heterogeneous porous media.

\bibliography{main}

\end{document}